\documentclass[11pt]{article}
\usepackage{graphicx}
 \usepackage{mathptmx,color}      
\usepackage[titletoc,title]{appendix}
\usepackage[numbers]{natbib}
\usepackage[T1]{fontenc}
\usepackage[applemac]{inputenc}
\usepackage[all]{xy}
\usepackage{enumerate}
\usepackage{comment}
\usepackage{amsmath,amssymb,stmaryrd}
\usepackage[british]{babel}
\usepackage{a4wide}
\usepackage[sans]{dsfont}
\usepackage{amsfonts}
\usepackage{amsthm}
\usepackage[mathscr]{euscript}
\usepackage{graphicx}
\usepackage{mathrsfs}
\usepackage{latexsym}
\usepackage{setspace}
\usepackage{amsthm,amsmath}

\newcommand{\p}{^}

\newcommand{\wt}{\widetilde}
\newcommand{\ol}{\overline}
	
	\newcommand{\lsi}{\llbracket}

	\newcommand{\rmd}{\mathrm{d}}
\newcommand{\rme}{\mathrm{e}}

\newcommand{\E}{\mathrm{E}}

\newcommand{\R}{\mathrm{R}}

\newcommand{\Var}{\mathrm{Var}}

\newcommand{\Ascr}{\mathscr{A}}
\newcommand{\Bscr}{\mathscr{B}}

\newcommand{\Fscr}{\mathscr{F}}
\newcommand{\Gscr}{\mathscr{G}}
\newcommand{\Hscr}{\mathscr{H}}

\newcommand{\Mscr}{\mathscr{M}}

\newcommand{\Oscr}{\mathscr{O}}
\newcommand{\Pscr}{\mathscr{P}}

\newcommand{\Rscr}{\mathscr{R}}

\newcommand{\Mloc}{\Hscr\p1_{\mathrm{loc}}}

\newcommand{\et}{\eta}

\newcommand{\ep}{\varepsilon}

\newcommand{\Lm}{\Lambda}
\newcommand{\sig}{\sigma}

\newcommand{\om}{\omega}
\newcommand{\Om}{\Omega}
\newcommand{\cadlag}{c\`adl\`ag }

\newcommand{\Ebb}{\mathbb{E}}

\newcommand{\Fbb}{\mathbb{F}}

\newcommand{\Pbb}{\mathbb{P}}

\newcommand{\Rbb}{\mathbb{R}}

\newcommand{\Hbb}{\mathbb{H}}
\newcommand{\Gbb}{\mathbb{G}}
\newcommand{\Id}{\text{Id}}

\newcommand{\aPP}[2]{\ensuremath{\langle #1,#2 \rangle}}

\newtheorem{theorem}{Theorem}[section]
\newtheorem{lemma}[theorem]{Lemma}
\newtheorem{proposition}[theorem]{Proposition}
\newtheorem{corollary}[theorem]{Corollary}
\theoremstyle{definition}
\newtheorem{definition}[theorem]{Definition}

\newtheorem{cexample}[theorem]{Counterexample}
\newtheorem{assumption}[theorem]{Assumption}

\newtheorem{remark}[theorem]{Remark}

\numberwithin{equation}{section}

\def\R{\mathbb R}

\def\E{\mathbb E}
\def\P{\mathbb P}

\def\F{\mathbb F}

\def\calc{{\cal C}}

\def\calg{{\cal G}}

\def\calp{{\cal P}}

\def\calu{{\cal U}}

\def\call{{\cal L}}

\makeatletter
\def\timenow{\@tempcnta\time
\@tempcntb\@tempcnta
\divide\@tempcntb60
\ifnum10>\@tempcntb0\fi\number\@tempcntb
:\multiply\@tempcntb60
\advance\@tempcnta-\@tempcntb
\ifnum10>\@tempcnta0\fi\number\@tempcnta}
\makeatother

\title{Progressively Enlargement of  Filtrations and  Control Problems for  Step Processes}
\author{Elena Bandini$^1$, Fulvia Confortola$^2$, Paolo Di Tella$^3$}
\date{}

\begin{document}
\maketitle
\begin{abstract}
In the present paper we address stochastic optimal control problems for a step process $(X,\Fbb)$ under a progressive enlargement of the filtration.  The  global  information  is obtained adding to the reference filtration $\Fbb$
the point process $H=1_{[\tau,+\infty)}$. Here $\tau$ is a random time that can be regarded as the occurrence time of an external shock event. 
 We  study two
classes of control problems, over $[0,T]$ and  over the random horizon $[0,T \wedge \tau]$.

 We solve these control problems following a dynamical approach based on a  class of BSDEs driven  by the jump measure $\mu\p Z$ of the semimartingale $Z=(X,H)$, which is a step process with respect to the enlarged filtration $\mathbb G$. The BSDEs that we consider can be solved in  $\Gbb$ thanks to a martingale representation theorem which we also establish here. To solve the BSDEs and the control problems we need to ensure that $Z$ is quasi-left continuous in the enlarged filtration $\Gbb$. Therefore, in addition to the  $\Fbb$-quasi left continuity of $X$, we assume some further conditions on $\tau$: the {\it avoidance} of $\Fbb$-stopping times  and the {\it immersion} property, or alternatively {\it Jacod's absolutely continuity} hypothesis.
 \end{abstract}
 {\noindent
\footnotetext[1]{ \emph{Adress:} Universit\`a di Bologna, Dipartimento di Matematica. Piazza di Porta San Donato 5, 40126  Bologna.
\emph{E-Mail: }{\tt elena.bandini7{\rm@}unibo.it}}

 {\noindent
\footnotetext[2]{ \emph{Adress:} Politecnico di Milano, Dipartimento di Matematica. Via E. Bonardi 9, 20133 Milano.
\emph{E-Mail: }{\tt fulvia.confortola{\rm@}polimi.it}}

{\noindent
\footnotetext[3]{ \emph{Adress:} Inst.\ f\"ur Mathematische Stochastik, TU Dresden. Zellescher Weg 12-14 01069 Dresden, Germany.
\emph{E-Mail: }{\tt Paolo.Di\_Tella{\rm@}tu-dresden.de}}

\noindent{\small\textbf{Keywords:} Progressive enlargement of filtration,  stochastic optimal control, marked point processes, Backward Stochastic Differential Equations (BSDEs).
}

\medskip 

\noindent{\small\textbf{MSC 2010:} 93E20, 60G55, 60G57.}

\section{Introduction}\label{sec:intro}

 The enlargement of filtrations is an important subject in probability theory. Starting from the first works published in the 80s (see e.g. Jeulin \cite{Jeu80} and Jacod \cite{J85}), this topic has received growing interest  due to its
application in credit risk research to model default events. The reference filtration, which describes the default-free information structure, is enlarged by the knowledge of a default time when it occurs,
leading to a global filtration, called progressive enlargement of filtration. Usually, it assumes that the credit event should arrive by surprise, i.e. it is a totally inaccessible
random time for the reference filtration, see e.g. Bielecki and Rutkowski \cite{BR02}, Duffie  and Singleton \cite{DS03}, and  Sch\"{o}nbucher \cite{S03}.
In this framework  various  stochastic optimal control problems can naturally arise. Here the notion of information plays a crucial role. The decision-maker studies a dynamical systems subject to random
perturbations and can act on it in order to optimize some performance criterion. In line with intuition, the more information the controller has available, the better his choice will be.

In the present work we build a bridge between the theory of enlargement of filtrations and control problems
for marked point processes.
More precisely, let $X$ be a step process, that is, a semimartingale $X$ of the form $X=\sum_{0\leq s\leq\cdot}\Delta X_s$  with finitely many jumps over compact time intervals.
We denote by $\Fbb$  the filtration generated by $X$:  This filtration represents the information available in a market in which an agent (i.e., the controller) handles. We progressively enlarge $\F$ to $\Gbb$ by  a random time $\tau$, that  can be regarded as the occurrence time of an external shock event, as the death of the agent (e.g., life insurance) or the default of part of the market (e.g., credit risk).
We then study two related classes of control problems.
The first one consists in  optimization problems over $[0,T]$. They  can be regarded as control problems of an insider trader who has private information on $\tau$, that is, she can use $\Gbb$-predictable controls. Moreover, the control problem over $[0,T]$ allows to consider terminal costs which may depend on the default time $\tau$, i.e., \emph{defaultable  costs}, for example of the form $g=g_11_{\{\tau>T\}}+g_21_{\{\tau\leq T\}}$, see also Remark \ref{rem:int1}. 
It is evident that the associated control problem cannot be solved in the reference filtration $\Fbb$ because the random variable $g$ {is $\Gscr_T$-measurable}  but not $\Fscr_T$-measurable, in general. 
The second class of control problems we look at is over the random horizon  $[0,T\wedge\tau]$. These can be understood as  control problems of an agent  who only disposes of the information available in the market, that is, she only uses $\Fbb$-predictable controls but, for some reasons, she has exclusively access to the market up to $\tau$, see also Remark \ref{R:2}.
{In this case the current costs are assumed to be $\Fbb$-predictable, while the terminal condition  will be $\Gscr_{T\wedge\tau}$-measurable}.

 Stochastic  optimal control problems have been intensively studied over  these  last  years, enspired especially  by economics and finance, where they  arise in risk management, portfolio selection or optimal investment. Several approaches have been developed ranging from the dynamic programming  method,  Hamilton- Jacobi-Bellman  Partial  Differential  Equations    and Backward  Stochastic  Differential  Equations  (BSDEs)  to  convex  martingale  duality  methods. We refer  to the monographs of Fleming and Soner \cite{FS06} and   Yong and  Zhou \cite{YZ99}. In particular, the  theory  of  BSDEs  has  known  important  developments also in the progressive enlarged filtration setting. For instance in Pham \cite{P10} an
 optimal investment problem for an agent delivering the defaultable claim at maturity T is solved with the BSDE approach. Similar problems have been addressed in Lim and Quenez \cite{LQ08}, Ankirchner, Blanchet-Scalliet, and Eyraud-Loisel \cite{ABE09}, and Jiao and   Pham \cite{JP09}.

Other results concern  expected utility optimization problems with random terminal time, where  the progressively enlarged filtration  is used to handle the time $\tau$ of a shock that affects the market or the
agent. These problems  can be solved by introducing a
suitable BSDE over $[0, T\wedge\tau]$.
In Kharroubi and Lim \cite{KL12} and Kharroubi, Lim and Ngoupeyou \cite{KLN13}, the authors consider optimization problems on $[0, T\wedge\tau]$ in a progressively enlarged Brownian filtration $\Gbb$ reducing the study of the BSDEs on $[0, T\wedge\tau]$ to the one of an associated BSDEs with deterministic horizon $[0,T]$ in the reference Brownian filtration $\Fbb$. This method is often called in the literature reduction method. Following the approach of \cite{KLN13}, Jeanblanc et al.\ study in \cite{JMPR15} an exponential utility maximization problem over $[0, T\wedge\tau]$ in a progressively enlarged Brownian filtration.
In both \cite{KLN13} and \cite{JMPR15} the random time $\tau$ avoids stopping times and satisfies the immersion property.
More general BSDEs over $[0, T\wedge\tau]$ have been considered  by Aksamit, Lim and Rutkowski in the recent preprint \cite{ALR21} where $\tau$ is a random time satisfying some mild conditions.
If from one side the reduction approach, as shown in \cite{ALR21}, allows very general random times $\tau$, it only allows to solve BSDEs in $\Gbb$ on $[0,T\wedge\tau]$ and not over $[0,T]$.

In the present paper,  
we show that the problem over $[0,T\wedge\tau]$ in $\Gbb$ can be solved as a problem in $\Gbb$ over the deterministic horizon $[0,T]$: We prove that the two value functions coincide, and there is an explicit relationship between the two optimal control processes, see Theorem \ref{P:5.22}.  This can be done   thanks to a martingale representation theorem in $\Gbb$  holding for all martingales and not only for martingales stopped at $\tau$, so that the BSDEs on $[0,T]$ can be globally solved. To our opinion this is for step processes the most natural approach since, as shown below, a martingale representation theorem always holds in the progressively enlargement $\Gbb$. 
We remark that this approach was initiated by Di Tella in \cite{DT19}, where an expected utility maximization problem for a continuous price processes was considered.
Thanks to our method, we are also able to consider defaultable costs and we can interpret the two control problems over $[0,T]$ and $[0,T\wedge\tau]$ as explained above. We stress that in the present paper we are able to consider a more general random time $\tau$ in comparison to \cite{DT19}: The intensity of $\tau$ need not be bounded here and we  do not {necessarily} require the immersion property. Furthermore, 
 we can also allow the presence of  running costs, which were not present in \cite{DT19}.

We now give more details on the control problems we  deal with.
They consist in minimizing a cost functional $J$ defined as the expectation of a cumulated running cost  plus a terminal cost. They are  formulated in a weak way, namely the expectation appearing in $J$ is taken with respect to a probability measure which depends on the control process itself.
We solve these control problems following a dynamical approach based on a  class of BSDEs with  Lipschitz generators of sub linear growth and driven by a particular random measure.
More precisely, since the filtration $\Gbb$ is obtained adding to the reference filtration $\Fbb$ the point process $H=1_{[\tau,+\infty)}$, we consider a class of BSDEs driven by the jump measure $\mu\p Z$ of the semimartingale $Z=(X,H)$, which is a $\Gbb$-marked point process.
Existence and  uniqueness of the solution of the involved BSDEs relay on the theory developed by Confortola and Fuhrman in \cite{CF13}.
 In \cite{CF13} the authors assume that the driving  marked point process is quasi-left continuous, i.e., its compensator is continuous.  This assumption is crucial: Indeed, Confortola, Fuhrman and Jacod have shown in \cite[Remark 10.1]{CFJ16} that if the compensator is not continuous, one cannot expect that the corresponding  BSDE admits a solution, in general.
On the other side, Di Tella and Jeanblanc  gave in \cite{DTJ21} a counterexample showing that  the $\Fbb$-quasi-left continuity of the marked point process $\mu\p X$ is not preserved in $\Gbb$, that is, $\mu\p Z$ need not be $\Gbb$-quasi-left continuous, see   \cite[Counterxample 4.7]{DTJ21}.
 Furthermore, in order to deal with a treatable class of control problems   we need the explicit form for the $\Gbb$-compensator $\nu\p {\Gbb, Z}$ of $\mu\p Z$, that in general is a challenging open problem.

For these reasons, in Section \ref{sec:PRP} we  assume sufficient conditions on $\tau$, namely  the {\it avoidance} of $\Fbb$-stopping times, see Definition \ref{def:ass.A} and the {\it immersion} property, see Definition \ref{def:hp.H}, or alternatively, {\it Jacod's absolutely continuity} hypotheses, see formulae \eqref{ass:jac.av}-\eqref{ass:jac}.
 Under these assumptions,  we are able to explicitly compute the $\Gbb$-compensator $\nu\p {\Gbb, Z}$ and show the $\Gbb$-quasi-left continuity of $\mu \p Z$, see Theorems \ref{thm:qlc.Z.A.H} and \ref{thm:com.G}. To the best of our knowledge, the explicit form  of $\nu\p {\Gbb, Z}$ under Jacod's absolutely continuity hypothesis is new and of independent interest.

Finally, to apply the theory for BSDEs developed in \cite{CF13}, we also need a martingale representation theorem in the enlarged filtration $\Gbb$.
{In Theorem \ref{thm:wrp.G}} we provide a martingale representation theorem {when} the filtration $\Fbb$ generated by a step process $X$ {is enlarged} by the filtration $\Hbb$ generated by a step process $H$. Taking $H=1_{[\tau,+\infty)}$, we obtain the special setting of the control problem. The results of Section \ref{sec:PRP} generalize the paper \cite{DTJ21} by Di Tella and Jeanblanc, in which $X$ and $H$ are \emph{simple point processes}, i.e.\ $\Delta X,\Delta H\in\{0,1\}$. In the {recent} preprint \cite{CT21} by Calzolari and Torti a martingale representation theorem for progressively enlarged filtrations generated by marked point processes is {also} obtained.  The results of \cite{CT21} are very general and go beyond the semimartingales context, while here we give a concise independent proof of the martingale representation theorem in the special case of step processes.

The present paper has the following structure: In Section \ref{sec:bas} we recall same basic notions. In Section \ref{sec:PRP} we obtain a martingale representation theorem in the enlarged filtration of a step process. Section \ref{sec:pred.proj} is devoted to the study of the $\Gbb$-quasi-left continuity of the step process $Z=(X,H)$, if $H=1_{[\tau,+\infty)}$, and to the computations for getting the explicit form of the compensator $\nu\p {\Gbb, Z}$. The applications to different kinds of control problems both on $[0,T]$ and $[0,T\wedge\tau]$ are presented in Section \ref{sec:appl.con},  while the proofs of technical results of Sections 4 and 5 are postponed respectively to Appendices \ref{sec:app} and \ref{A:B}.

\section{Basic Notions}\label{sec:bas}

Let $(\Om,\Fscr,\Pbb)$ be a probability space. We denote by $\Fbb=(\Fscr_t)_{t\geq0}$ a right-continuous filtration of subsets of $\Fscr$ and by $\Oscr(\Fbb)$ (resp.\ $\Pscr(\Fbb)$) the $\Fbb$-optional (resp.\ $\Fbb$-predictable) $\sig$-algebra on $\Om\times\Rbb_+$. We define $\Fscr_\infty:=\bigvee_{t\geq0}\Fscr_t$.

Let $X$ be a stochastic process. We sometimes use the notation $(X,\Fbb)$ to mean that $X$ is $\Fbb$-adapted. By $\Fbb\p X$ we denote the smallest right-continuous filtration such that $X$ is adapted. If $X$ is \cadlag, we denote by $\Delta X$ the jump process and use the convention $\Delta X_0=0$.

We say that an $\Fbb$-adapted \cadlag\ process $X$ is $\Fbb$-quasi-left continuous if $\Delta X_\sig=0$ a.s.\ for every finite-valued $\Fbb$-predictable stopping time $\sig$.

\textbf{Random measures.} For a Borel subset $E$ of $\Rbb\p d$, we introduce $\wt\Om:=\Om\times\Rbb_+\times E$ and the product $\sig$-algebras $\wt\Oscr(\Fbb):=\Oscr(\Fbb)\otimes\Bscr(E)$ and $\wt\Pscr(\Fbb):=\Pscr(\Fbb)\otimes\Bscr(E)$. If $W$ is an $\wt\Oscr(\Fbb)$-measurable (resp.\ $\wt\Pscr(\Fbb)$-measurable) mapping from $\wt\Om$ into $\Rbb$, it is called an $\Fbb$-optional (resp.\ $\Fbb$-predictable) function.

Let $\mu$ be a random measure  on $\Rbb_+\times E$ (see \cite[Definition II.1.3]{JS00}).
For a nonnegative $\Fbb$-optional function $W$, we write $W\ast\mu=(W\ast\mu_t)_{t\geq0}$, where $W\ast\mu_t(\om):=\int_{[0,t]\times E}W(\om,s,x)\mu(\om,\rmd s,\rmd x)$ is the process defined by the (Lebesgue--Stieltjes) integral of $W$ with respect to $\mu$ (see \cite[II.1.5]{JS00} for details). If $W\ast\mu$ is $\Fbb$-optional (resp.\ $\Fbb$-predictable), for every optional (resp.\ $\Fbb$-predictable) function $W$, then $\mu$ is called $\Fbb$-optional (resp.\ $\Fbb$-predictable).

\textbf{Semimartingales.}
 When we say that $X$ is a semimartingale, we always assume that it is \cadlag. For an $\Rbb\p d$-valued $\Fbb$-semimartingale $X$, we denote by $\mu\p X$ the jump measure of $X$, that is,
$
\mu\p{X}(\om,\rmd t,\rmd x)=\sum_{s>0}1_{\{\Delta X_s(\om)\neq0\}}\delta_{(s,\Delta X_s(\om))}(\rmd t,\rmd x)$, where, here and in the whole paper, $\delta_a$ denotes the Dirac measure at point $a$.
From \cite[Theorem II.1.16]{JS00}, $\mu\p X$ is an \emph{integer-valued random measure} on $\Rbb_+\times \Rbb\p d$ with respect to $\Fbb$ (see \cite[Definition II.1.13]{JS00}). Thus, $\mu\p X$ is, in particular, an $\Fbb$-optional random measure. According to \cite[Definition III.1.23]{JS00}, $\mu\p X$ is called an $\Rbb\p d$-valued \emph{marked point process} (with respect to $\Fbb$) if $\mu\p X(\om;[0,t]\times \Rbb\p d)<+\infty$, for every $\om\in\Om$ and $t\in\Rbb_+$. By $\nu\p X$ we denote the $\Fbb$-predictable compensator of $\mu\p X$ (see \cite[Definition II.2.6]{JS00}). We recall that $\nu\p X$ is a predictable random measure characterized by the following properties: For any $\Fbb$-predictable function $W$ such that $|W|\ast\mu\p X\in\Ascr\p +_\mathrm{loc}$, we have $|W|\ast\nu\p X\in\Ascr\p +_\mathrm{loc}$ and $W\ast\mu\p X-W\ast\nu\p X\in\Mloc(\Fbb)$, $\Mloc(\Fbb)$ denoting the space of $\Fbb$-local martingales and $\Ascr\p +_\mathrm{loc}(\Fbb)$ the space of $\Fbb$-adapted locally integrable \cadlag \, increasing processes starting at zero.

If $X=Y-Z$ with $Y,Z\in\Ascr\p +_\mathrm{loc}(\Fbb)$, we then write $X\in\Ascr_\mathrm{loc}(\Fbb)$. For $X\in\Ascr_\mathrm{loc}(\Fbb)$ we denote by $X\p{p,\Fbb}$ the $\Fbb$-dual predictable projection of $X$, that is the unique $\Fbb$-predictable process in $\Ascr_\mathrm{loc}(\Fbb)$ such that $X-X\p{p,\Fbb}\in\Mloc(\Fbb)$.

An $\Rbb\p d$-valued semimartingale $X$ is a step process with respect to $\Fbb$ if  it can be represented in the form $X=\sum_{n=1}\p\infty\xi_n1_{[\tau_n,+\infty)}$, where $(\tau_n)_n$ is a sequence of $\Fbb$-stopping times such that $\tau_n\uparrow+\infty$, $\tau_n<\tau_{n+1}$ on $\{\tau_n<+\infty\}$ and $(\xi_n)_{n\geq1}$ is a sequence  of $\Rbb\p d$-valued random variables such that $\xi_n$ is $\Fscr_{\tau_n}$-measurable  and $\xi_n\neq0$ if and only if $\tau_n<+\infty$ (see \cite[Definition 11.48]{HWY92}). The process $N\p X=\sum_{n=1}\p\infty1_{[\tau_n,+\infty)}$ is called \emph{the point process associated to} $X$.  If $X$ is a step process with respect to $\Fbb$, we then obviously have $\tau_n=\inf\{t>\tau_{n-1}: X_t\neq X_{\tau_{n-1}}\}$ ($\tau_0:=0$), $\xi_n=\Delta X_{\tau_n}1_{\{\tau_n<+\infty\}}$ and
\begin{equation}\label{eq:st.pr.j.mea}
\mu\p X(\rmd t,\rmd x)=\sum_{n=1}\p\infty 1_{\{\tau_n<+\infty\}}\delta_{(\tau_n,\xi_n)}(\rmd t,\rmd x).
\end{equation}

 We say that a semimartingale $X$ is a sum of jumps with respect to $\Fbb$ if $X$ is $\Fbb$ adapted, of finite variation and $X=\sum_{0\leq s\leq\cdot}\Delta X_s$. If $X$ is a sum of jumps, then  $X=\Id\ast\mu\p X$ holds, where $\Id(x):=x$. Furthermore, it
is evident, that $\mu\p X$ is a marked step process if and only if $X$ is a step process (see \cite[III.1.21 and Proposition II.1.14]{JS00}), that is, if  $X$ has finitely many jumps over compact time intervals.

\section{Martingale Representation}\label{sec:PRP}

For an $\Rbb\p d$-valued step process $(X,\Fbb\p X)$ and a $\sig$-field $\Rscr\p X$, called the \emph{initial} $\sig$-\emph{field}, we denote by $\Fbb=(\Fscr_t)_{t\geq0}$ the filtration $\Fbb\p X$ initially enlarged by $\Rscr\p X$, thus $\Fscr_t:=\Rscr\p X\vee\Fscr_t\p X$.  It is well-known that $\Fbb$ is right-continuous and clearly, $(X,\Fbb)$ is a step process. We stress that  non-trivial initial $\sigma$-field $\Rscr\p X$ allows to include in the theory developed in the present paper, without any additional effort, also step processes  with $\Fscr_0$-measurable semimartingale characteristics,  that is, step-processes with conditionally independent increments with respect to $\Fbb$ given $\Fscr_0$.

We now consider an $\Rbb\p\ell$-valued step process $H$ and introduce $\Hbb=(\Hscr_t)_{t\geq0}$ by $\Hscr_t:=\Rscr\p H\vee\Fscr\p H_t$, $t\geq0$, where $\Rscr\p H$ denotes a $\sig$-field.

 The \emph{progressive enlargement} of $\Fbb$ by $\Hbb$ we denote by $\Gbb=(\Gscr_t)_{t\geq0}$, where
\[
 \Gscr_t:=\bigcap_{s>t} \Fscr_{s}\vee\Hscr_{s} \quad t\geq0.
\]
 It is evident that $\Gbb$ is the \emph{smallest} right-continuous filtration containing $\Fbb\p X$, $\Fbb\p H$, $\Rscr\p X$ and $\Rscr\p H$ (i.e., $\Fbb$ and $\Hbb$.

 As a special example of $H$, one can take the default process associated with a random time $\tau$, i.e., $H_t(\om):=1_{\lsi \tau,+\infty\lsi}(\om,t)$, where $\tau$ is a $(0,+\infty]$-valued random variable. In this case $(H,\Hbb)$ is a point process. If $\Rscr\p H$ is trivial, $\Gbb$ is called the progressive enlargement of $\Fbb$ by $\tau$ and it is the smallest right-continuous filtration containing $\Fbb$ and such that $\tau$ is a $\Gbb$-stopping time.

We now introduce the $\Rbb\p d\times\Rbb\p\ell$-valued $\Gbb$-semimartingale $Z=(X,H)\p\top$. Clearly, $Z$ is a sum of jumps with respect to $\Gbb$, hence it is a $\Gbb$-semimartingale. The jump measure $\mu\p{Z}$ of $Z$ is an integer-valued random measure on $\Rbb_+\times \Rbb\p d\times\Rbb\p\ell$ and satisfies
\[
\mu\p{Z}(\om,\rmd t,\rmd x_1,\rmd x_2)=\sum_{s<0}1_{\{\Delta Z_s(\om)\neq0\}}\delta_{(s,\Delta Z_s(\om))}(\rmd s,\rmd x_1,\rmd x_2).
\]

\begin{theorem}\label{thm:wrp.G}
Let $(X,\Fbb\p X)$  and $(H,\Fbb\p H)$ be step processes taking values in $\Rbb\p d$ and $\Rbb\p\ell$ respectively and consider two initial $\sig$-fields $\Rscr\p X$ and $\Rscr\p H$. We define the filtrations $\Fbb$, $\Hbb$ and $\Gbb$ as above and set $Z:=(X,H)\p\top$. We then have:

\textnormal{(i)} $\mu\p Z$ is an $\Rbb\p d\times\Rbb\p\ell$-valued marked point process, that is $(Z,\Gbb)$ is a step process.

\textnormal{(ii)} $\Gbb$ is the smallest right-continuous filtration containing $\Rscr:=\Rscr\p X\vee\Rscr\p H$ and such that $\mu\p Z$ is optional.

\noindent If furthermore $\Fscr=\Gscr_\infty$, then every $Y\in\Mloc(\Gbb)$ can be represented as
\begin{equation}\label{eq:wrp.G}
Y=Y_0+W\ast\mu\p Z-W\ast\nu\p Z
\end{equation}
where $(\om,t,x_1,x_2)\mapsto W(\om,t,x_1,x_2)$ is a $\Pscr(\Gbb)\otimes\Bscr(\Rbb\p d)\otimes\Bscr(\Rbb\p\ell)$-measurable function such that $|W|\ast\mu\p Z\in\Ascr\p+_{\mathrm{loc}}(\Gbb)$ and  $\nu\p Z$ denotes the $\Gbb$-dual predictable projection of the jump measure $\mu\p Z$ of $Z$.
\end{theorem}
\begin{proof}
We start proving (i). Since $Z$ is a sum of jumps, it is sufficient to show that $\mu\p Z$ is a marked point process with respect to $\Gbb$. To this aim, we observe that $\mu\p X$ and $\mu\p H$ are an $\Rbb\p d$-valued and an $\Rbb\p \ell$-valued marked point process with respect to $\Gbb$, respectively, $(X,\Gbb)$ and $(H,\Gbb)$ being an $\Rbb\p d$-valued and an $\Rbb\p \ell$-valued step processes, respectively. Therefore, we have \[
\begin{split}
\mu\p Z([0,t]\times\Rbb\p d\times\Rbb\p\ell)
=
\sum_{0\leq s\leq t}1_{\{\Delta Z_s\neq0\}}
&=
\sum_{0\leq s\leq t}1_{\{\Delta X_s\neq0\}\cup\{\Delta H_s\neq0\}}
\\&\leq
\sum_{0\leq s\leq t}(1_{\{\Delta X_s\neq0\}}+1_{\{\Delta H_s\neq0\}})
\\&=\mu\p X([0,t]\times\Rbb\p d)+\mu\p H([0,t]\times\Rbb\p\ell)<+\infty,
\end{split}
\]
meaning that $\mu\p Z$ is an $\Rbb\p d\times\Rbb\p \ell$-valued marked point process with respect to $\Gbb$. This concludes the proof of (i). We now come to (ii). Let us denote by $\Gbb\p\prime$ the smallest right continuous filtration such that $\mu\p{Z}$ is optional. We first show the identity $\Gbb\p\prime=\Fbb\p{Z}$. Since $Z$ is an $\Fbb\p{Z}$-semimartingale, $\mu\p{Z}$ is an $\Fbb\p{Z}$-optional integer-valued random measure. So, $\Gbb\p\prime\subseteq\Fbb\p{Z}$ holds. We now show the converse inclusion $ \Gbb\p\prime\supseteq\Fbb\p{}$. We denote
$g_1(x_1,x_2)=x_1$ and $g_2(x_1,x_2)=x_2$. By definition of $\mu\p{Z}$ we have
\[
\begin{split}
|g_1|\ast\mu\p{Z}_t&=\sum_{s\leq t}|\Delta X_s|1_{\{\Delta Z\neq0\}}
\\&=
\sum_{s\leq t}|\Delta X_s|\big(1_{\{\Delta X_s\neq0,\Delta H_s=0\}}+1_{\{\Delta X_s=0,\Delta H_s\neq0\}}+1_{\{\Delta X_s\neq0,\Delta H_s\neq0\}}\big)
\\&=
\sum_{s\leq t}|\Delta X_s|1_{\{\Delta X_s\neq0\}}\leq\Var(X)_t<+\infty,
\end{split}
\]
 where $\Var(X)_t(\om)$ denotes the total variation of $s\mapsto X_s(\om)$ on $[0,t]$. Hence, the integral $g_1\ast\mu\p{Z}$ is well defined and satisfies $X=g_1\ast\mu\p{Z}$.  Analogously,  $H=g_2\ast\mu\p{Z}$ holds. This yields that $X$ and $H$ are $\Gbb\p\prime$-optional processes. Since $\Gbb\p\prime$ is right-continuous, we get $\Gbb\p\prime\supseteq\Fbb\p{Z}$. From \cite[Proposition 3.39 (a)]{J79} the filtration $\Rscr\vee\Gbb\p\prime$ is right-continuous. Therefore,  $\Rscr\vee\Gbb\p\prime$  and $\Gbb$ coincide: They are both the smallest right continuous filtrations containing $\Fbb\p X$, $\Fbb\p H$, $\Rscr\p X$ and $\Rscr\p H$. The proof of (ii) is complete. We now come to \eqref{eq:wrp.G}. If we assume $\Fscr=\Gscr_\infty$, this is an immediate consequence of (i), (ii) and  \cite[Theorem III.4.37]{JS00}. The proof is complete.
\end{proof}
As an application of Theorem \ref{thm:wrp.G}, we can easily show by induction the following result.
\begin{corollary}\label{cor:seq.en}
We consider the $\Rbb\p{d_i}$-valued step processes $(X\p i,\Fbb\p{X\p i})$ and the initial $\sigma$-fields $\Rscr\p{X\p i}$, $i=1,\ldots,n$. We set $\Fbb\p i:=\Fbb\p{X\p i}\vee\Rscr\p{X\p i}$ and denote by $\Gbb$ the smallest right-continuous filtration containing $\{\Fbb\p i,\ i=1,\ldots,n\}$. Then the $E:=\Rbb\p{d_1}\times\cdots\times\Rbb\p{d_n}$-valued semimartingale $Z=(X\p1,\ldots, X\p n)\p\top $ satisfies:

\textnormal{(i)} $\mu\p Z$ is an $E$-valued marked point process with respect to $\Gbb$, that is $Z$ is an $E$-valued step process with respect to $\Gbb$.

\textnormal{(ii)} $\Gbb$ is the smallest right continuous filtration containing $\Rscr:=\bigvee_{i=1}\p n\Rscr\p{X\p i}$ and such that $\mu\p Z$ is an optional random measure.

\noindent If furthermore $\Fscr=\Gscr_\infty$, then every $Y\in\Mloc(\Gbb)$ can be represented as
\begin{equation}\label{eq:wrp.Gn}
Y=Y_0+W\ast\mu\p Z-W\ast\nu\p Z
\end{equation}
where $(\om,t,x_1,\ldots,x_n)\mapsto W(\om,t,x_1,\ldots,x_n)$ is a $\Pscr(\Gbb)\otimes\Bscr(E)$-measurable function such that $|W|\ast\mu\p Z\in\Ascr\p+_{\mathrm{loc}}(\Gbb)$ and  $\nu\p Z$ denotes the $\Gbb$-dual predictable projection of the jump measure $\mu\p Z$ of $Z$.
\end{corollary}

\section{The $\Gbb$-dual predictable projection}\label{sec:pred.proj}
Consider an $\Rbb\p\ell$-valued step process $H$ and introduce $\Hbb=(\Hscr_t)_{t\geq0}$ by $\Hscr_t:=\Rscr\p H\vee\Fscr\p H_t$, $t\geq0$, where $\Rscr\p H$ denotes a $\sig$-field. We denote by $\Gbb=(\Gscr_t)_{t\geq0}$ the progressive enlargement of $\Fbb$ by $\Hbb$.

In this section we shall concentrate on the special case $H=1_{[\tau,+\infty)}$, where $\tau$ denotes a random time: In particular, we aim to establish sufficient conditions to ensure the $\Gbb$-quasi-left continuity of $Z=(X,H)$. To this end we need to determine the form of the $\Gbb$-dual predictable projection $\nu\p{\Gbb,Z}$ of $\mu\p Z$. The next result, which holds for general step processes, gives the structure of $\nu\p{\Gbb,Z}$ if $H$ and $X$ have no common jumps. Its proof is postponed to the Appendix \ref{sec:app}.

\begin{theorem}\label{T:Zcomp}
 Let $(X, \Fbb)$ be  an $\R^d$-valued step-process and let $(H, \Hbb)$ be an $\R^\ell$-valued step-process. If $\Delta X \Delta H=0$, then  the following identities hold for the $\R^d\times\R\p\ell$-valued $\Gbb$-step process $Z=(X,H)$:
	
		\textnormal {(i)} $\mu^Z(\om, \rmd t, \rm x_1, \rmd x_2)=\mu^X(\om, \rmd t, \rmd x_1) \delta_0(\rmd x_2) +\mu^H(\om, \rmd t, \rmd x_2) \delta_0(\rmd x_1)$.
		
		\textnormal {(ii)} $\nu^{\Gbb,Z}(\om, \rmd t, \rmd x_1, \rmd x_2)= \nu^{\Gbb,X}(\om, \rmd t, \rmd x_1)\delta_0(\rmd x_2)+\nu^{\Gbb,H}(\om, \rmd t, \rmd x_2)\delta_0(\rmd x_1)$.
\end{theorem}

\subsection{Progressive enlargement by a random time}\label{sec:prog.en.rt}
We now denote by $\Gbb$ the progressive enlargement of $\Fbb$ by a random time $\tau:\Om\longrightarrow(0,+\infty]$, that is, $\Gbb$ is the smallest right-continuous filtration containing $\Fbb$ and such that $\tau$ is a $\Gbb$-stopping time.

We denote by $H=1_{[\tau,+\infty)}$ the default process of $\tau$ and by $A={}\p o(1-H)={}\p o1_{[0,\tau)}$ the $\Fbb$-optional projection of $(1-H)=1_{[0,\tau)}$  (see \cite[Theorem V.14 and V.15]{D72}). The process $A$ is a  \cadlag  $\Fbb$-supermartingale, called Az\'ema supermartingale, satisfying $A_t=\Pbb[\tau>t|\Fscr_t]$ a.s., for every $t\geq0$. Let $m$ be the martingale defined by $m_t=\Ebb[H\p o+1_{\{\tau=+\infty\}}|\Fscr_t]$ a.s., $t\geq0$, where $H\p o$ denotes the $\Fbb$-dual optional projection of $H$. The martingale $m$ belongs to the class $BMO$ with respect to $\Fbb$ (see \cite[Proposition 1.49]{AJ17}) and $A$ has the following $\Fbb$-optional decomposition: $A=m-H\p o$. It is well-known that $\{A_->0\}\subseteq[0,\tau]$ holds (see, e.g., \cite[Lemma 2.14]{AJ17}), so the process $\frac1{A_-}1_{[0,\tau]}$ is well defined. For every $\Fbb$-local martingale $Y$, the process $[Y,m]$ belongs to $\Ascr_\mathrm{loc}(\Fbb)$.  Therefore, the $\Fbb$-dual predictable projection $[Y,m]\p{p,\Fbb}$ of $[Y,m]$ is well defined and we set $\aPP{Y}{m}\p{\Fbb}:=[Y,m]\p{p,\Fbb}$.

 The $\Gbb$-dual predictable projection of $H$ is denoted by $\Lm\p\Gbb$ and, by \cite[Proposition 2.15]{AJ17} it satisfies
\begin{equation}\label{LambdaG}
	\Lm\p\Gbb=\int_0\p{\tau\wedge\cdot}\frac{1}{A}_{s-}\rmd H_s\p{p,\Fbb},
\end{equation}
$H\p{p,\Fbb}$ denoting the $\Fbb$-dual predictable projection of $H$.

Because of the special structure of the enlarged filtration, the following result holds, whose proof is given in   Appendix \ref{sec:app}.
\begin{lemma}\label{lem:dec.pred.fun}
Let  $(\om,t,x)\mapsto W(\om,t,x)$ be a $\Gbb$-predictable function. Then, there exists an $\Fbb$-predictable function $(\om,t,x)\mapsto \ol W(\om,t,x)$ such that $W(\om,t,x)1_{[0,\tau]}(\om,t)=\ol W(\om,t,x)1_{[0,\tau]}(\om,t)$. If furthermore $W$ is bounded, then $\ol W$ is bounded too.
\end{lemma}

\subsection{Quasi-left continuity in the enlarged filtration}\label{subs:qlc.en}
 Let $X\in\Ascr_\mathrm{loc}(\Fbb)$  be an $\Fbb$-quasi left continuous step process. We are interested in the following question: Is $X$ $\Gbb$-quasi left continuous? The reason to consider this problem is that to solve the BSDEs associated to the control problems studied in Section \ref{sec:appl.con}, we want to apply the theory developed in \cite{CF13} and, to this aim, it is important that the $\Gbb$-step process $Z=(X,H)$ is $\Gbb$-quasi-left continuous.
However, in general, this is not true: Intuitively, the larger filtration $\Gbb$ supports more predictable stopping times than $\Fbb$. To see this we recall \cite[Counterexample 4.8]{DTJ21}:
\begin{cexample}\label{cex:no.qlc}  Let $X$ be a homogeneous Poisson process with respect to $\Fbb\p X$ and let $(\tau_n)_{n\geq 1}$ be the sequence of the jump-times of $\Fbb\p X$. The process $X$ is not quasi-left continuous in the filtration $\Gbb$ obtained enlarging $\Fbb\p X$ progressively by the random time $\tau=\frac{1}{2}(\tau_1+\tau_2)$. Indeed, the jump-time $\tau_2$ of $X$ is announced in $\Gbb$ by $(\vartheta_n)_{n\geq1}$, $\vartheta_n:=\frac1n\tau+(1-\frac1n)\tau_2$, and $\vartheta_n>\tau$ is a $\Gbb$-stopping time for every $n\geq1$ by \cite[Theorem III.16]{D72}. Hence, $\tau_2$ is a $\Gbb$-predictable jump-time of $X$.
\end{cexample}
Therefore, we need to state sufficient conditions of $\tau$ to ensure the $\Gbb$-quasi-left continuity of $Z$. Notice that the quasi-left continuity of $X$ can get lost only over $(\tau,+\infty]$, as the following result shows. The proof is given in  Appendix \ref{sec:app}.
\begin{proposition}\label{prop:qlc.stop}
Let $X\in\Ascr_\mathrm{loc}(\Fbb)$ be an $\Fbb$-quasi-left continuous step process. Then, the $\Gbb$-adapted stopped process $X\p\tau$ (defined by $X\p\tau_t:=X_{t\wedge\tau}$, $t\geq0$) is $\Gbb$-quasi-left continuous.
\end{proposition}

In the remaining part of this section we state sufficient conditions for the $\Gbb$-quasi-left continuity of $Z$.

\paragraph{Avoidance of $\Fbb$-stopping times. } The first property we are going to recall is the \emph{avoidance of $\Fbb$-stopping times}, from now on referred as assumption $(\Ascr)$.
\begin{definition}\label{def:ass.A}
The random time $\tau$ satisfies assumption $(\Ascr)$ if $\Pbb[\tau=\sig<+\infty]=0$ for every $\Fbb$-stopping time $\sig$.
\end{definition}
The interpretation of assumption $(\Ascr)$ is the following: The random time $\tau$ carries an information which is completely exogenous: Nothing about $\tau$ can be inferred from the information contained in the reference filtration $\Fbb$. In particular, $\tau$ satisfies ($\Ascr$) if and only if $H\p{o}$ is continuous (see, e.g., \cite[Lemma 1.48(a)]{AJ17} or \cite[Lemma 3.4]{DTE21}). This implies that $H\p{o}=H\p{p,\Fbb}$ and, according to \eqref{LambdaG}, we get that $\Lm\p\Gbb$ is a continuous process. Therefore, $H$ is a quasi-left continuous process and $\tau$ is a totally inaccessible stopping time.

If $\tau$ satisfies assumption $(\Ascr)$, then $\Delta X\Delta H=0$. Indeed,
\begin{equation}\label{eq:cov.HX}
[X,H]_t=\sum_{s\leq t}\Delta X_s\Delta H_s=\sum_{s\leq t}\Delta X_s\Delta H_s1_{\{\Delta X_s\neq0\}\cap\{\Delta H_s\neq0\}}.
\end{equation}
Denoting by $(\sig_n)_{n\geq1}$ a sequence exhausting the thin set $\{\Delta X\neq0\}$ we obviously have $\{\Delta X\neq0\}\cap\{\Delta H\neq0\}=\bigcup_{n=1}\p\infty[\sig_n]\cap[\tau]$, where for a stopping time $\et$ we denote by $[\et]$ the graph of $\et$. Because of assumption $(\Ascr)$, the random set $[\sig_n]\cap[\tau]$ is evanescent, for every $n\geq1$. Hence, \eqref{eq:cov.HX} yields $[X,H]=0$ and therefore $\Delta X\Delta H=\Delta[X,H]=0$.
Thus, according to Theorem \ref{T:Zcomp}, if assumption $(\Ascr)$ is satisfied, by the special form of $H$, we have
\begin{align}
\mu^Z(\om, \rmd t, \rm x_1, \rmd x_2)&=\mu^X(\om, \rmd t, \rmd x_1) \delta_0(\rmd x_2) +\rmd H_t(\om)\delta_1(\rmd x_2)\delta_0(\rmd x_1),\notag\\
		\nu^{\Gbb,Z}(\om, \rmd t, \rmd x_1, \rmd x_2)&= \nu^{\Gbb,X}(\om, \rmd t, \rmd x_1)\delta_0(\rmd x_2)+\rmd\Lm\p\Gbb_t(\om)\delta_1(\rmd x_2)\delta_0(\rmd x_1).\label{nuGZLambda}
\end{align}

\paragraph{Immersion property.} As we have seen above, if $\tau$ satisfies assumption $(\Ascr)$, then the process $H$ is $\Gbb$-quasi-left continuous. This however, does not imply that the joint process $Z=(X,H)$ is $\Gbb$-quasi-left continuous. Indeed, the random time $\tau$ from Counterexample \ref{cex:no.qlc} avoids $\Fbb$-stopping times. However, $X$ is not $\Gbb$-quasi-left continuous. So, we need further assumptions to ensure the $\Gbb$-quasi-left continuity of $Z$. To this aim we state the following definition
\begin{definition}\label{def:hp.H}
We say that the random time $\tau$ satisfies the immersion property, from now on assumption $(\Hscr)$, if $\Fbb$-martingales remain $\Gbb$-martingales.
\end{definition}
If $X$ is an $\Fbb$-step process and $\tau$ is a random time satisfying assumption $(\Hscr)$, then \cite[Theorem 2.21]{JS00} yields $\nu\p{\Gbb,X}=\nu\p{\Fbb,X}$. Therefore, if $X$ is $\Fbb$-quasi-left continuous, i.e.,  $\nu\p{\Fbb,X}(\om,\{t\}\times\Rbb\p d)=0$, $t\geq0$, the same holds in $\Gbb$. We therefore have the following result, whose proof follows from the above discussion and is therefore omitted.
\begin{theorem}\label{thm:qlc.Z.A.H}
 Let $X$ be a step process and let $\tau$ be a random time satisfying both assumption $(\Ascr)$ and assumption $(\Hscr)$. Then, the $\Gbb$-dual predictable projection of $\mu\p Z$, where $Z=(X,H)$, is given by
\[\begin{split}
\nu^{\Gbb,Z}(\om, \rmd t, \rmd x_1, \rmd x_2)&= \nu^{\Fbb,X}(\om, \rmd t, \rmd x_1)\delta_0(\rmd x_2)+\delta_1(\rmd x_2)\delta_0(\rmd x_1)\rmd\Lm\p\Gbb_t(\om).
\end{split}
\]
In particular, if $X$ is $\Fbb$-quasi-left continuous, then $Z$ is $\Gbb$-quasi-left continuous as well.
\end{theorem}
We stress that, although assumption $(\Hscr)$ is of technical nature, it is equivalent to require that the $\sig$-fields $\Fscr_\infty$ and $\Gscr_t$ are conditionally independent given $\Fscr_t$ (see \cite[Theorem 3.2]{AJ17}). Furthermore, because of the Cox construction (see \cite[\S2.3.1]{AJ17}), it is easy to construct random times $\tau$ satisfying the assumptions $(\Ascr)$ and $(\Hscr)$ (see \cite[Remark 3.8]{DTE21} for details).

\paragraph{Jacod's absolute continuity condition.}
Let $\et$ denote the law of the random time $\tau$, that is, $\et(A):=\Pbb(\tau\p{-1}(A))$, for every $A\in\Bscr(\Rbb)$. We denote by $P_t(\om,A)$ a regular version of the conditional distribution $\Pbb[\tau\in A|\Fscr_t]$, $A\in\Bscr(\Rbb)$.
We make the following assumptions.
\begin{align}
&\et\ \text{ \emph{is a diffused probability measure}.}\label{ass:jac.av}\\
&P_t(\rmd u)\ \text{\ \emph{is absolutely continuous with respect to} }\ \et(\rmd u).\label{ass:jac}
\end{align}
 We stress that random times of this type can be constructed following the approach presented by Jeanblanc and Le Cam in \cite[\S5]{JLC09}.

 We say  that a random time $\tau$ with \eqref{ass:jac.av} and \eqref{ass:jac} satisfies Jacod's absolute continuity condition.
	According to \cite[Corollary 2.2]{EJJ10}, \eqref{ass:jac.av} ensures property ($\Ascr$).
Thanks to \cite[Proposition 4.17]{AJ17}, \eqref{ass:jac} implies that there exists a nonnegative and $\Oscr(\Fbb)\otimes\Bscr([0,+\infty])$-measurable function $(\om,t,u)\mapsto p_t(\om,u)$ such that $(p_t(u))_{t\geq0}$ is an $\Fbb$-martingale for every $u\in[0,+\infty]$ and
\begin{equation}\label{eq:con.exp}
\Ebb[f(\tau)|\Fscr_t]=\int_\Rbb f(u)p_t(u)\et(\rmd u),\quad t\geq0
\end{equation}
for every bounded Borel function $f$.

We stress that Jacod's absolute continuity does not imply, in general, that $\tau$ satisfies assumption $(\Hscr)$. This is only true if and only if $p_t(u)=p_u(u)$ $\et$-a.s., if $u<t$ (see \cite[Proposition 5.28]{AJ17}).

 In the next result we give the form of the $\Gbb$-dual predictable projection of $\mu\p X$. To the best of our knowledge, this result is new and of independent interest.

\begin{theorem}\label{thm:com.G}
Let $(X,\Fbb)$ be an $\Fbb$-quasi left continuous step process, where $\Fbb:=\Fbb\p X\vee\Rscr$, and let $\tau$ be a random time satisfying  \eqref{ass:jac.av} and \eqref{ass:jac}. Then the $\Gbb$-dual predictable projection $\nu\p{\Gbb,X}$ of $\mu\p X$ is given by
\begin{equation}\label{eq:com.G}
\nu\p{\Gbb, X}(\om,\rmd t,\rmd x)=\Big(1_{[0,\tau]}(\om,t)\Big(1+\frac{W\p\prime(\om,t,x)}{A_{t-}}\Big)+1_{(\tau,+\infty)}(\om,t)(1+U(\om,t,x))\Big)\nu\p{\Fbb, X}(\om,\rmd t,\rmd x)
\end{equation}
where $W\p\prime$ is an $\Fbb$-predictable function such that $A_-+W\p\prime\geq0$ and $U$ is a $\Gbb$-predictable function such that $1+U\geq0$ identically. In particular, $X$ is a $\Gbb$-quasi left continuous step process.
\end{theorem}
 The proof of Theorem \ref{thm:com.G} is technical and we give it in  Appendix \ref{sec:app}. At this point we only observe that the function $U$ in Theorem \ref{thm:com.G} can be constructed in the following way, as in the proof of \cite[Proposition 3.14 and Theorem 4.1]{J85}. Let $p_t(u)$ be the process appearing in \eqref{eq:con.exp}.  Since $p_\cdot(u)$ is an $\Fbb$-martingale for every $u$, we can represent it as $p_\cdot(u)=W\p u\ast\mu\p X-W\p u\ast\nu\p{\Fbb,X}$. The function $(\om,t,x,u)\mapsto W\p u(\om,t,x)$ can be chosen $\wt\Pscr(\Fbb)\otimes\Bscr([0,+\infty])$-measurable and such that $W\p u(\om,t,x)+p_{t-}(\om,u)\geq0$ and $p_{t-}(u)=0\Longrightarrow W\p u(\om,t,x)=0$. So using the convention $\frac00:=0$, one can define $V\p u(\om,t,x):=\frac{W\p u(\om,t,x)}{p_{t-}(\om,u)}$ which satisfies $1+V\p u\geq1$ identically. The $\Gbb$-predictable function $U$ is then given by \[U(\om,t,\om)=V\p{\tau(\om)}(\om,t,x)1_{(\tau,+\infty)}(\om,t)=\frac{W\p {\tau(\om)}(\om,t,x)}{p_{t-}(\om,\tau(\om))}1_{(\tau,+\infty)}(\om,t).\]

The following result is a direct application of \eqref{nuGZLambda},
Theorem \ref{thm:com.G}, and \eqref{LambdaG}.
\begin{corollary}\label{C:mainres}
	Let $\tau$ be a random time satisfying  conditions   \eqref{ass:jac.av} and \eqref{ass:jac}. Then
\begin{align*}
	&\nu^{\Gbb,Z}(\om, \rmd t, \rmd x_1, \rmd x_2)\\
	&= \Big(1_{[0,\tau]}(\om,t)\Big(1+\frac{W\p\prime(\om,t,x_1)}{A_{t-}(\om)}\Big)+1_{(\tau,+\infty)}(\om,t)(1+U(\om,t,x_1))\Big)\nu\p{\Fbb, X}(\om,\rmd t,\rmd x_1)\delta_0(\rmd x_2)\\
	&+ 1_{[0,\tau]}(\om, t)\frac{1}{A_{t-}(\om)}\rmd H_t\p{p,\Fbb}(\om)\delta_0(\rmd x_1)\delta_1(\rmd x_2),
\end{align*}
where $W\p\prime$ is an $\Fbb$-predictable function such that $A_-+W\p\prime\geq0$ and $U$ is a $\Gbb$-predictable function such that $1+U\geq0$ identically.
\end{corollary}

 We remark that if $\tau$ is a random time satisfying  conditions   \eqref{ass:jac.av} and \eqref{ass:jac}, then Corollary  \ref{C:mainres} yields the $\Gbb$-quasi left continuity of the step process $Z$, whenever $X$ is $\Fbb$-quasi left continuous.

\section{Applications to stochastic control theory}\label{sec:appl.con}

In this section we consider an optimization problem for marked point processes, in presence of an additional exogenous risk source that cannot be inferred from the information available in the market, represented by the filtration $\Fbb$.  The additional
risk source can be a shock event, as the death of the investor or the default of part of the
market. Its occurrence time is modelled by $\tau$ . We will solve it by means of suitable BSDEs, relying on the theory developed  in \cite{CF13}. The fundamental tool in order to apply the results in \cite{CF13} to the present context  is Theorem \ref{thm:wrp.G}.

Let $T>0$ be a fixed  finite time horizon.
 Let $X$ be an $\Rbb\p d$-valued step process, and set $\Fbb=(\Fscr_t)_{t\geq0}$ where $\Fscr_t:=\Fscr_t\p X$. In particular, $\Fscr_0$ is trivial, since $X_0=0$, $X$ being a step process.
Let $\tau: \Om \rightarrow (0,+\infty]$ be a random time and  $H$ be the default process associated with $\tau$. We denote by $\Gbb$ the progressive enlargement of $\Fbb$ by $\tau$, and by $\Gbb_T =(\mathcal G_t)_{t \in [0,T]}$  the restriction of $\Gbb$ to $[0,T]$.
We also introduce the   jump measure $\mu\p X$ of $X$ and the corresponding $\Fbb$-dual predictable projection $\nu\p {\mathbb{F}, X}$.

Introduce the step process $Z=(X,H)$ with jump measure  $\mu^Z$ and corresponding $\Gbb$-dual predictable projection $\nu^{\Gbb,Z}$.
By Theorem \ref{thm:wrp.G}, we know that $Z$ satisfies the WRT with respect to $\Gbb_T$.

To ensure that the theory developed in \cite{CF13} can be applied to the enlarged filtration $\Gbb$ we now state the following assumptions.

\begin{assumption}
\label{A:appl}
 The process $X$ is $\Fbb$-quasi-left continuous.
\end{assumption}
\begin{assumption}\label{A:appl2}
 $\tau$ satisfies ($\Ascr$) and ($\Hscr$).
\end{assumption}
\begin{assumption}\label{A:appl3}
 $\tau$ satisfies \eqref{ass:jac.av}-\eqref{ass:jac} (hence ($\Ascr$)).
\end{assumption}

In the following we will always assume Assumption \ref{A:appl} together with Assumption \ref{A:appl2} \emph{or alternatively} with Assumption \ref{A:appl3}.

\begin{remark}
(i) The  $\Fbb$-compensator $\nu\p{\Fbb,X}$ of $\mu^X$  admits the decomposition
\[
\nu^{\Fbb,X}(\rmd t,\,\rmd x_1)=
\phi^{\Fbb, X}_t(\rmd x_1)\rmd C^{\Fbb, X}_t,
\] where  $\phi^{\Fbb, X}$ is a transition probability from  $(\Om\times [0,T], \mathcal P(\Fbb))$ into $(\Rbb^d, \mathcal B(\Rbb^d))$, and, by   Assumption \ref{A:appl},  $C^{\Fbb, X}\in\Ascr\p +_\mathrm{loc}(\Fbb)$ is a continuous process. Analogously,  the  $\Gbb$-compensator $\nu\p{\Gbb,X}$ of $\mu^X$  admits the decomposition
$$
\nu^{\Gbb,X}(\rmd t,\,\rmd x_1)=
\phi^{\Gbb, X}_t(\rmd x_1)\rmd C^{\Gbb, X}_t,
$$
where  $\phi^{\Gbb, X}$ is a transition probability from  $(\Om\times [0,T], \mathcal P(\Gbb))$ into $(\Rbb^d, \mathcal B(\Rbb^d))$, with  $C^{\Gbb, X}\in\Ascr\p +_\mathrm{loc}(\Gbb)$.  If furthermore Assumption \ref{A:appl} together with Assumption \ref{A:appl2} \emph{or} with Assumption \ref{A:appl3} hold, by Theorem \ref{thm:qlc.Z.A.H} or by Theorem \ref{thm:com.G}, $C^{\Gbb, X}$ is also a continuous process.
	\\[.2cm](ii)
	From Assumption \ref{A:appl} and condition ($\Ascr$) it follows that the
$\sigma$-algebra $\mathcal G_0$ is trivial.
\end{remark}

\begin{remark}\label{R:contC}
Because of $(\Ascr)$, by Theorem \ref{T:Zcomp} 
we get that
\[
\mu^Z(\om, \rmd t, \rmd x_1, \rmd x_2)=\mu^X(\om, \rmd t, \rmd x_1) \delta_0(\rmd x_2) + \rmd H_t(\om)\delta_1(\rmd x_2)\delta_0(\rmd x_1)
\]
and the corresponding $\Gbb$-dual predictable projection is given by
\begin{align}\label{nuZ2}
	\nu^{Z}(\om, \rmd t, \rmd x_1, \rmd x_2)=  \delta_0(\rmd x_2)\phi^{\Gbb, X}_t(\om, \rmd x_1)\rmd C^{\Gbb, X}_t(\om)
	+
	\delta_0(\rmd x_1)\delta_1(\rmd x_2)
	\rmd \Lambda_t^{\Gbb}(\om).
\end{align}
By  Theorem \ref{thm:qlc.Z.A.H} (if Assumption \ref{A:appl2} holds) or Theorem \ref{thm:com.G} (if Assumption \ref{A:appl3} holds) we have that $Z$ is a $\Gbb$-quasi left continuous step process.
\end{remark}

We notice that the random measure $\nu\p Z$ in \eqref{nuZ2} can be rewritten as
	\begin{align}\label{nuZbis}
	\nu^{Z}(\om, \rmd t, \rmd x_1, \rmd x_2)= \phi_t(\om, \rmd x_1, \rmd x_2)\rmd C_t( \om),
\end{align}
where  $C$ is defined by
\begin{align}\label{C}
C_t(\omega):=
C^{\Gbb, X}_t(\om)+ \Lambda_t^{\Gbb}(\om)
\end{align}
and $\phi$  is a transition probability from  $(\Om\times [0,T], \Pscr(\Gbb))$ into $(\Rbb^{d+1}, \mathcal B(\Rbb^{d+ 1}))$.

We also observe that, by \eqref{nuZ2}, the following identities  hold:
\begin{align*}
	1_{\{x_2 =0\}} \,\nu^{Z}(\om, \rmd t, \rmd x_1, \rmd x_2) &=  \delta_0(\rmd x_2)\phi^{\Gbb, X}_t(\om, \rmd x_1)\rmd C^{\Gbb, X}_t(\om),\\
	1_{\{x_2 \neq 0\}} \,\nu^{Z}(\om, \rmd t, \rmd x_1, \rmd x_2) &= \delta_1(\rmd x_2)\delta_0(\rmd x_1)
		\rmd \Lambda_t^{\Gbb}.
\end{align*}
In particular,  integrating previous expressions on $\R^{d+1}$, we get
\begin{align}
	d_1(\omega, t) \rmd C_t(\omega) &= \rmd C^{\Gbb, X}_t(\om),\label{f:CG}\\
	d_2(\omega, t)  \rmd C_t(\omega) &=
		\rmd \Lambda_t^{\Gbb},\label{f:CG2}
\end{align}
with
\begin{align}
	d_1(\omega, t)&:= \int_{\R^{d+ 1}} 1_{\{x_2 =0\}} \,\phi_t(\omega, \rmd x_1, \rmd x_2),\label{d1}\\
	d_2(\omega, t)&:=\int_{\R^{d+ 1}} 1_{\{x_2 \neq 0\}} \,\phi_t(\omega, \rmd x_1, \rmd x_2).\label{d2}
\end{align}

\begin{remark}\label{equivdec}
Under Assumptions \ref{A:appl}-\ref{A:appl2} we have
$\phi^{\Gbb, X}_t(\om, \rmd x_1)=\phi^{\F, X}_t(\om, \rmd x_1)$,
	$\rmd C^{\Gbb, X}_t(\om)=\rmd C^{\F, X}_t(\om)$
while  under Assumptions \ref{A:appl}- \ref{A:appl3} one gets
\begin{align*}
	\phi^{\Gbb, X}_t(\om, \rmd x_1)&=\frac{\kappa(\omega, t, x_1)}{\int_{\R^d}\kappa(\omega, t, x_1)\,\phi^{\F, X}_t(\om, \rmd x_1)}\phi^{\F, X}_t(\om, \rmd x_1), \\
\rmd C^{\Gbb, X}_t(\om)&=\int_{\R^d}\kappa(\omega, t, x_1)\,\phi^{\F, X}_t(\om, \rmd x_1)\,\rmd C^{\F, X}_t(\om),
\end{align*}
where
$$
\kappa(\omega, t, x_1):=1_{[0,\tau]}(\om,t)\Big(1+\frac{W\p\prime(\om,t,x_1)}{A_{t-}(\om)}\Big)+1_{(\tau,+\infty)}(\om,t)(1+U(\om,t,x_1)\Big)
$$
is the density appearing in Theorem \ref{thm:com.G}.
\end{remark}

\subsection{The control problem on  [0,\,T]}\label{S:T}

The data specifying the optimal control problem are an
  action  space
$U$, a running cost function $l$, a terminal cost function $g$,
and another function $r$ specifying the effect of the control
process. They are assumed to satisfy the following conditions.

\begin{assumption}\label{hyp:controllosingle}
$(U,\calu)$ is a measurable space.
\end{assumption}

\begin{assumption}\label{hyp:rl}
 The functions $r,l:\Omega\times [0,T]\times \R^{d}\times U\to \R$
are $\calp(\Gbb)\otimes \mathcal B(\R^{d})\otimes \calu$-measurable and
 there exist
constants $M_r> 1$, $M_l>0$ such that, $\P$-a.s.,
\begin{equation}\label{ellelimitato}
0\le r_t (x_1,u)\le M_r,\quad |l_t (x_1,u)|\le M_l,\qquad t\in [0,T],
x_1 \in \R^d,
 u\in U.
\end{equation}
\end{assumption}

\begin{assumption}\label{hyp:controllo_g}
 The function $g:\Omega\times \R^{d}\to \R$ is $\calg_T\otimes
\mathcal B(\R^{d})$-measurable, and there exists a constant  $\beta$ such that $\beta > \sup|r-1|^2$, and
\begin{align}
 &\E [\rme^{ \beta C_T}] < + \infty, \label{Cint}\\
 &\E [|g(X_T)|^2 \rme^{\beta C_T} ] < + \infty. \label{gint}
 \end{align}
\end{assumption}
To every  admissible control process $u \in \mathcal C$ we will associate the cost functional
\begin{equation}\label{cost_T}
 J(u)=\E_u\left[
\int_0^Tl_t(X_t,u_t)\rmd C^{\Gbb, X}_t + g(X_T)\right],
\end{equation}
where $\E_u$ denotes the expectation under a probability measure $\P_u$, absolutely continuous with respect to $\P$, that will be specified below.
The control problem will consists in minimizing $J$ over all the admissible controls.  Because of the structure of the control problem, it is evident that in general it cannot be solved in the filtration $\Fbb$. Therefore,  we have to allow $\Gbb$-predictable strategies: The set of admissible
control processes, denoted $\mathcal C$, consists of all $U$-valued and $\Gbb$-predictable processes $u(\cdot)=(u_t)_{t\in[0,T]}$.

\begin{remark}[Interpretation]\label{rem:int1} We now give an interpretation of the control problem associated with \ref{cost_T}. One	could consider for instance a defaultable terminal cost $g:\Omega\times \R^{d}\to \R$ of the form $g(\omega, x_1)=g_1(x_1)1_{\{T<\tau(\omega)\}}+g_2(x_1)1_{\{T\geq\tau(\omega)\}}$.
 In this case, $g_1$ is the terminal  cost if the default does not occur before the maturity $T$ while $g_2$ is the cost to pay in case of default up to maturity. Similarly, one can allow a defaultable running cost $l$.  So, the agent acting in the enlarged market can minimize the cost functional also over defaultable terminal costs and running costs. This is not possible when acting in the reference market represented by $\Fbb$. One could also regard the minimization problem associated to \eqref{cost_T} as the problem of an insider who disposes of private information about $\tau$ and whose strategies are  $U$-valued and $\Gbb$-predictable.
\end{remark}

Related to every control $u\in{\mathcal C}$, we  introduce the predictable random measure
\begin{align}\label{nuZu}
	\nu^{Z,u}(\om, \rmd t, \rmd x_1, \rmd x_2)&=r_t(\omega, X_t(\omega),   u_t(\omega))\,\delta_0(\rmd x_2)\,\phi^{\Gbb, X}_t(\om, \rmd x_1) \rmd C^{\Gbb, X}_t(\om)\notag\\
&+ 
	\delta_1(\rmd x_2)\delta_0(\rmd x_1)
	\rmd \Lambda_t^{\Gbb}.
\end{align}
Let us now set
\begin{equation}\label{tilder}
	R_t(x_1, x_2,  u) := r_t(x_1,  u)\,1_{\{x_2 =0\}} + \,1_{\{x_2 \neq 0\}}, \quad x_1 \in \R^d, \,\,x_2 \in \{0,1\}, \,\,u \in U.
\end{equation}
By \eqref{f:CG}-\eqref{f:CG2},  we have $\nu^{Z,  u}=R_t(X_t,H_t,   u_t)\nu^Z$.

 We denote by $(T_n)_{n\geq1}$ the sequence of jump times of $ Z$ and, for any $u\in{\mathcal C}$, we consider the process
\begin{equation}\label{Lubis}
L^{ u}_t=
\exp\Big(\int_0^t\int_{\R^{d+ 1}} (1-R_s (x_1,x_2, u_s))\nu^Z( \rmd s, \rmd x_1, \rmd x_2)\Big)
\prod_{n\ge1\,:\,T_n\le t}R_{T_n} (X_{T_n},H_{T_n},u_{T_n}),
\end{equation}
with the convention that the last product equals $1$ if there are
no indices $n\ge 1$ satisfying $T_n\le t$. We notice that $L^{ u}$ is a Dol\'eans-Dade stochastic exponential, solution to the equation
\[
L\p u_t=1+\int_0\p t\int_{\Rbb\p{d+1}}L\p u_{s-}(R_s(x_1,x_2,u_s)-1)(\mu\p Z-\nu\p Z)(\rmd s,\rmd x_1,\rmd x_2).
\]
Hence, $L\p u$ is a $\Gbb$-local martingale, for every $u\in\calc$. Furthermore, $L\p u$ is nonnegative (see \cite[Proposition 4.3]{J74} for details), thus it is a nonnegative supermartingale.

Taking into account  \eqref{tilder}, we remark that
\begin{align*}
	&\int_0^t\int_{\R^{d+ 1}} (1-R_s (x_1,x_2, u_s))\nu^Z( \rmd s, \rmd x_1, \rmd x_2)\\
	&= \int_0^t\int_{\R^{d}} \int_{\R} (1-(r_s(x_1, u))\delta_0(\rmd x_2)\phi^{\Gbb, X}_s( \rmd x_1) \rmd C^{\Gbb, X}_s
\end{align*}
so that  \eqref{Lubis} reads
\begin{equation}\label{Lu}
L^{ u}_t=
\exp\Big(\int_0^t\int_{\R^{d}} (1-r_s (x_1, u_s))\phi^{\Gbb, X}_s( \rmd x_1) \rmd C^{\Gbb, X}_s\Big)
\prod_{n\ge1\,:\,T_n\le t}(r_{T_n} (X_{T_n},u_{T_n})\,1_{\{H_{T_n} =0\}}
+\,1_{\{H_{T_n} \neq 0\}}).
\end{equation}
The  result below follows from   \cite[Lemma 4.2]{CF13} with
$\gamma = 2$.
\begin{lemma}\label{L:Lu}
Assume that
\begin{equation}\label{C_r}
\E [\rme^{(3 + M^4_r )C_T}] < + \infty.
\end{equation}
Then, for every $u \in \calc$, 	$\sup_{t \in [0,\,T]} \E [|L^{ u}_t|^2] < \infty$ and $\E [L^{u}_T ] = 1$.  In particular, $L\p u$ is a square integrable $\Gbb$-martingale for every $u\in\calc$.
\end{lemma}

 By Lemma \ref{L:Lu} we can define
an absolutely continuous probability measure $\P_{u}$
by  setting
\[\P_{u}(d\omega)=L_T^{u}(\omega)\P(d\omega).\]
It can then be proven (see e.g.\
\cite[Theorem 4.5]{J74})  that the $\Gbb$-compensator $\nu^{Z,u}$ of $\mu^Z$
under $\P_{u}$ is given by \eqref{nuZu}.
To every $u\in\calc$ we can then associate the cost functional
\eqref{cost_T},
where $\E_u$ denotes the expectation under $\P_u$.
The control problem is

\begin{equation}\label{contr_pb}
\inf_{u \in \mathcal C} J(u)=\inf_{u \in \mathcal C} \E_u\left[
\int_0^Tl_t(X_t,u_t)\rmd C^{\Gbb, X}_t + g(X_T)\right].
\end{equation}
Notice that $J$
in \eqref{cost_T} is finite for every admissible control. Moreover,
$g(X_T)$ is integrable under $\P_u$, since
\begin{equation}\label{GintPu}
\E_u[|g(X_T)|] = \E[|L_T^u g(X_T)|] \leq (\E [|L_T^u|^2])^{1/2}(\E [|g(X_T)|^2])^{1/2}< \infty
\end{equation}
where the latter inequality follows from \eqref{gint} in  Assumptions   \ref{hyp:controllo_g}. Moreover, under  Assumption \ref{hyp:rl}  and recalling \eqref{hyp:controllo_g} and  \eqref{f:CG}, we get
$$
\E_u\left[
\int_0^Tl_t(X_t,u_t) \rmd C^{\Gbb, X}_t\right]=
\E_u\left[
\int_0^Tl_t(X_t,u_t)d_1(t) \rmd C_t\right]
\normalcolor\leq M_l \,\E_u[C_T] < \infty.
$$

\begin{remark}[The action of the insider]\label{rem:int2}
Because of \eqref{nuZu}, in the optimal control problem \eqref{contr_pb} the   insider
acts under $\P^u$ by  changing  the  $\Gbb$-compensator of $X$, while the one of $H$ (and hence of $\tau$) remains untouched.
\end{remark}

\textbf{The associated BSDE.}
We next proceed to the solution of the optimal control problem
formulated above. A fundamental role is played by  the following BSDE: $\P$-a.s., for all $t\in [0,T]$,
\begin{equation}\label{bsdecontrollo}
    Y_t+\int_t^T\int_{\R^{d+ 1}} \Theta_s(x_1,x_2)\,(\mu^Z- \nu^Z)(\rmd s,\rmd x_1,\rmd x_2) =
g(X_T) +\int_t^T f(s,X_s,\Theta_s(\cdot))\,\rmd C^{\Gbb, X}_s.
\end{equation}
The generator of BSDE \eqref{bsdecontrollo} is defined by means of
the Hamiltonian function
\begin{equation}\label{defhamiltonian}
    f(\omega,t,y_1,\theta(\cdot)):=\inf_{u\in U}\Big\{
l_t(\omega, y_1,u)+ \int_{\R^{d}} \theta(x_1,0) \, (r_t
(\omega,x_1,u)-1)\,\phi_t^{\Gbb, X}(\omega,\rmd x_1)\Big\},
\end{equation}
for every $\omega\in\Omega$, $t\in[0,T]$, $y_1 \in \R^d$, $y_2 \in \R^1$ and $\theta\in\call^1(\R^{d + 1},\mathcal B(\R^{d + 1}), \phi_t(\omega,\rmd x_1, \rmd x_2))$.

For $\beta >0$, we look for a solution  $(Y, \Theta(\cdot))$ to \eqref{bsdecontrollo} in the space $L^{2, \beta}_{\text{Prog}}(\Omega \times [0,\,T],\Gbb) \times L^{2, \beta}(\mu^Z,\Gbb)$,  where
$L^{2, \beta}_{\text{Prog}}(\Omega \times [0,\,T],\Gbb)$ denotes the set of real-valued $\mathbb G$-progressively measurable processes $Y$ such that
\begin{align*}
	\E\Big[\int_0^T \rme^{\beta C_t}|Y_t|^2\rmd C_t\Big]< \infty,
\end{align*}
and $ L^{2, \beta}(\mu^Z,\Gbb)$ denotes the set of $\mathcal P(\mathbb G) \otimes \mathcal B(\R^{d + 1})$-measurable functions $\Theta$ such that
\begin{align*}
	&\E\Big[\int_0^T \int_{\R^{d + 1}} \rme^{\beta C_t}|\Theta_t(x_1, x_2)|^2\phi_t(\rmd x_1, \rmd x_2)\rmd C_t\Big]\\
	&=\E\Big[\int_0^T \int_{\R^{d}} \rme^{\beta C_t}|\Theta_t(x_1, 0)|^2\phi^{\Gbb, X}_t(\rmd x_1)\rmd C^{\Gbb, X}_t\Big]+ \E\Big[\int_0^T  \rme^{\beta C_t}|\Theta_t(0,1)|^2
	\rmd \Lambda_t^{\Gbb}\Big]< \infty.
\end{align*}

By   $L^{1,0}(\mu^Z,\Gbb)$ we denote the set of $\mathcal P(\mathbb G) \otimes \mathcal B(\R^{d + 1})$-measurable functions $\Theta$ such that
\begin{align*}
	&\E\Big[\int_0^T \int_{\R^{d + 1}} |\Theta_t(x_1, x_2)|\phi_t(\rmd x_1, \rmd x_2)\rmd C_t\Big]\\
	&=\E\Big[\int_0^T \int_{\R^{d}}|\Theta_t(x_1, 0)|\phi^{\Gbb, X}_t(\rmd x_1)\rmd C^{\Gbb, X}_t\Big]+ \E\Big[\int_0^T |\Theta_t(0,1)|\,
	\rmd \Lambda_t^{\Gbb}\Big]< \infty.
\end{align*}
We note the inclusion $ L^{2,\beta}(\mu^Z)  \subseteq L^{1,0}(\mu^Z)$ for all $\beta>0$ holds  (see \cite[Remark 3.2-2.]{CF13}).
We will consider the following additional assumption:
\begin{assumption}\label{hyp:hamiltoniana}
For every $\Theta\in L^{1,0}(\mu^Z,\Gbb)$ there exists a $\Gbb$-predictable process (i.e., an admissible control)
$\underline{u}^\Theta:\Omega\times [0,T]\to U$,  such that, for $d_1(\omega,t)\rmd C_t(\om) \P(d\omega)$-almost all $(\omega, t)$, we have
\begin{align}\label{minselector}
    f(\omega,t,X_{t-}(\omega),\Theta_t(\omega, \cdot))&=
l_t(\om, X_{t-}(\omega),\underline{u}^\Theta(\omega,t )) \\
&+ \int_{\R^{d+ 1}} \Theta_t(\omega, x_1, 0) \,
(r_t
(\omega,x_1, \underline{u}^\Theta(\omega,t))-1)\,\phi_t^{\Gbb,X}(\omega,\rmd x_1).\notag
\end{align}
\end{assumption}
\begin{remark}\label{R:splitted}
Assumption \ref{hyp:hamiltoniana} can be verified in specific
situations when it is possible to compute explicitly the function
$\underline{u}^\Theta$. General conditions for its validity can also be
formulated using appropriate measurable selection theorems, as the case $U$ compact metric space, and $l_t(\om, x, \cdot), r_t(\om, x, \cdot): U \rightarrow \R$ continuous functions, see
\cite[Proposition 4.8]{CF13}.
\end{remark}

 Thanks to  the WRT  for  $Z$ with respect to $\Gbb_T$ provided in Theorem \ref{thm:wrp.G}, one can show the existence and the uniqueness of the solution of BSDE \eqref{bsdecontrollo}. The proof of the proposition below is postponed to Appendix \ref{A:B}.
\begin{proposition}\label{Pwellpos1}
	Let Assumptions \ref{A:appl}, \ref{hyp:controllosingle}, \ref{hyp:rl} and  \ref{hyp:hamiltoniana} hold true. Assume  that Assumption \ref{A:appl2} or \ref{A:appl3} holds true. Set
	\begin{align}\label{D:L}
 L:=\textup{ess}\sup_{\omega} \big( \sup\{ |r_t (x_1, u)-1|\,:\,
  t\in [0,T], \,x_1 \in \R^d, \,u\in U\} \big)
\end{align}
and let  Assumption
\ref{hyp:controllo_g}  hold true with $\beta >L^2$.
Then BSDE \eqref{bsdecontrollo} admits a unique solution $(Y, \Theta(\cdot))\in L^{2, \beta}_{\text{Prog}}(\Omega \times [0,\,T],\Gbb) \times L^{2, \beta}(\mu^Z,\Gbb)$. 
\end{proposition}

\textbf{Solution to the optimal control problem.}
At this point we can give	 the main result of the section.

\begin{theorem}\label{teoremacontrollononmarkov}
Let Assumptions \ref{A:appl}, \ref{hyp:controllosingle}, \ref{hyp:rl}  and
\ref{hyp:hamiltoniana} hold true. Assume also that Assumption \ref{hyp:controllo_g} holds true
with $\beta>L^2$, with $L$ in \eqref{D:L}, and  that condition \eqref{C_r} holds true. Let Assumption \ref{A:appl2} or \ref{A:appl3} holds true, and let  $(Y,\Theta)\in L^{2, \beta}_{\text{Prog}}(\Omega \times [0,\,T],\Gbb) \times L^{2, \beta}(\mu^Z,\Gbb)$  denote
 the unique  solution to  BSDE \eqref{bsdecontrollo}, with
 corresponding admissible control $\underline{u}^\Theta \in \calc$  satisfying \eqref{minselector}.
Then
$\underline{u}^\Theta$ is   optimal and $Y_0$ is the optimal cost, i.e.
$$
Y_0= J(\underline{u}^\Theta)= \inf_{u\in\calc }J(u).
$$
\end{theorem}
\proof
The proof 
consists in proving the so-called fundamental relation.
We first recall that, by Lemma \ref{L:Lu},  for every $u \in \calc$, we have	$\sup_{t \in [0,\,T]} \E [|L^{ u}_t|^2] < \infty$.
Moreover, by \eqref{GintPu}, $\Ebb_u[|g(X_T)|]<+\infty$.
Let $u \in \calc$ be fixed.
Then, H\"older inequality and Assumption \ref{hyp:rl} yield $\Theta(\cdot) \in L^{1,0}(\mu^Z,\Gbb)$ under $\P_u$.
Setting $t=0$ and taking the expectation $\E_u[\cdot]$ in BSDE \eqref{bsdecontrollo}, we get
\begin{align*}
    Y_0+\E_u\big[\int_0^T\int_{\R^{d}} \Theta(x_1,0) \, (r_s
(x_1,u)-1)\,\phi^{\Gbb, X}_s(\rmd x_1)\rmd C^{\Gbb, X}_s\big]=
\E_u[g(X_T)] +\E_u\big[\int_0^T f(s,X_s,\Theta_s(\cdot))\,\rmd C^{\Gbb, X}_s\big].
\end{align*}
Then, adding and subtracting $\E_u\big[\int_0^T l_s(X_s,u)\,\rmd C^{X,\Gbb}_s\big]$, we obtain
\begin{align*}
    &Y_0= J(u) \\
&+\E_u\Big[\int_0^T \big[f(s,X_{s},\Theta_s(\cdot))-l(s,X_{s},u) -\int_{\R^{d}} \Theta_s(x_1,0) \, (r_s
(x_1,u)-1)\,\phi^{\Gbb, X}_s(\rmd x_1) \big]\rmd C^{\Gbb, X}_s\Big]
\end{align*}
where we have also used the continuity of $C$. The conclusion follows from the definition of $f$ in \eqref{defhamiltonian}, noticing that the term in the square brackets is non positive, and it equals $0$ if $u(\cdot) = \underline{u}^\Theta(\cdot)$.
\endproof

\subsection{The control problem on $[0,\,T \wedge \tau]$}
We now consider the problem of an agent for whom the available information is exclusively given by $\Fbb$
(that is, she pursues $\Fbb$-predictable strategies) but, for some reasons, she has only access to the
market up to the occurrence of the exogenous shock event, whose occurrence time is modelled
by $\tau$.  For example, the problem over $[0,T\wedge\tau]$ can be regarded as the optimization problem of an agent who minimizes running costs not up to the maturity $T>0$, but only up to $T\wedge\tau$.

For simplicity we consider in this section only the case where Assumptions \ref{A:appl}-\ref{A:appl2} are satisfied, so that, according to Remark \ref{equivdec},   $\rmd C^{\F, X}=\rmd C^{\Gbb, X}$ and $\phi^{\F, X}(\rmd x_1)=\phi^{\Gbb, X}(\rmd x_1)$.
In this context, this seems  to be a natural assumption. Indeed, let Jacod's assumption hold for $\tau$. Then, by \cite[Corollary 3.1]{JLC09}, $\tau$ satisfies $(\Hscr)$ if and only if $p_\cdot(u)$ is constant after $u$, that is, $p_t(u)=p_t(t)$, $t\geq u$, a.s. As observed in \cite[p.1016]{EJJ10}, this is substantially equivalent to say that the ``\emph{information contained in the reference filtration after the default time gives no new information on the conditional distribution of the default}''. But, since we restrict our attention to $[0,T\wedge\tau]$, that is, before the default, we are neglecting all information after default.

We still consider a measurable space  $(U, \mathcal U)$ satisfying Assumption \ref{hyp:controllosingle}.
The other data specifying the optimal control problem are
a running cost function $\bar l$, a terminal cost function $\bar g$,
and a  function $\bar r$, that are assumed to satisfy the following conditions.

\begin{assumption}\label{hyp:bar_rl}
 The functions $\bar r,\bar l:\Omega\times [0,T]\times \R^{d}\times U\to \R$
are $\calp(\Fbb)\otimes \mathcal B(\R^{d})\otimes \calu$-measurable and
 there exist
constants $M_{\bar r}> 1$, $M_{\bar l}>0$ such that, $\P$-a.s.,
\begin{equation}\label{bar_ellelimitato}
0\le \bar r_t (x_1,u)\le M_{\bar r},\quad |\bar l_t (x_1,u)|\le M_{\bar l},\qquad t\in [0,T],\,
x_1 \in \R^d,\, u\in U.
\end{equation}
\end{assumption}
\normalcolor

\begin{assumption}\label{hyp:controllo_barg_sigma}
The function $\bar g:\Omega\times \R^{d}\to \R$ is $\calg_{T \wedge \tau}\otimes
\mathcal B(\R^{d})$-measurable, and there exists a constant $\beta$ such that $\beta > \sup|\bar r-1|^2$,  and
\begin{align}
&\E [\rme^{ \beta C_{T}}] < + \infty  \label{Ctau 2}, \\
&\E [|\bar g(X_{T \wedge \tau})|^2 \rme^{\beta C_{T}} ] < + \infty.\label{g tau 2}
 \end{align}
\end{assumption}

Let $\calc$ be the set of admissible strategies for the optimization problem
 introduced in Section \ref{S:T}. For any $u \in \calc$, we define $\hat{u}:= 1_{[0,T \wedge \tau ]}u$. Clearly
$\hat u \in \calc$  holds true. We define now
the new set  of admissible strategies as
\begin{equation}\label{hatA}
\hat{\calc}:= \{ u \in \calc : 1_{[T \wedge \tau ,T]}u= 0\} \subseteq \calc.
\end{equation}
 Since $\Fbb$-predictable and $\Gbb$-predictable processes coincide on $[0, \tau ]$ (see \cite[Lemma 4.4. b)]{Jeu80}), the set  $\hat{\calc}$
given in \eqref{hatA}  consists of strategies which are morally
 $\Fbb$-predictable.

To every control $\hat u\in\hat{\calc}$, we  associate the predictable random measure
$
	\nu^{Z,\hat u}(\om, \rmd t, \rmd x_1, \rmd x_2)$ of the same form of \eqref{nuZu}.
We have $\nu^{Z, \hat u}=(\bar r_t(X_t,  \hat u_t)d_1 + d_2)\nu^Z$, where $d_1$ and $d_2$ are the densities in \eqref{d1}-\eqref{d2}.

For any $\hat u\in\hat{\calc}$, under condition \eqref{C_r} with $M_{r}$ replaced by $M_{\bar r}$, we can  consider then    the Dol\'{e}ans-Dade exponential
 martingale  $L^{\hat u}$  in \eqref{Lu}, and we can introduce the absolutely continuous probability measure $\P_{\hat u}$ defined as
$\P_{\hat u}(\rmd \omega)=L_T^{\hat u}(\omega)\P(\rmd \omega)$.
We then consider a cost functional of the form
\begin{equation}\label{J2bis}
\bar J(\hat u)= \E_{\hat u}\left[
\int_0^{T\wedge \tau} \bar l_t(X_t,\hat u_t)\rmd C^{\F, X}_t + \bar g(X_{T\wedge \tau})\right],\quad \hat u \in \hat{\calc},
\end{equation}
where $\E_{\hat u}$ denotes the expectation under $\P_{\hat u}$.
The control problem is now
\begin{equation}\label{contr_pb_2}
\inf_{\hat u \in \hat {\mathcal C}} \bar J(\hat u)=\inf_{\hat u \in \hat {\mathcal C}} \E_{\hat u}\left[
\int_0^{T\wedge \tau} \bar l_t(X_t,\hat u_t)\rmd C^{\F, X}_t + \bar g(X_{T\wedge \tau})\right].
\end{equation}

\begin{remark}\label{R:2}
The  control problem in \eqref{contr_pb_2} can be interpreted as the one of an agent who only controls $X$ using $\Fbb$-predictable strategies but only up to the occurrence $\tau$ of an external risky event. Hence, because of the exogenous risk source, this control problem cannot be solved in $\Fbb$.
\end{remark}

\textbf{The associated BSDE.}
The optimal control problem in \eqref{contr_pb_2} can be solved by means of the following BSDE:
  $\P$-a.s., for all $t\in [0,T]$,

	\begin{align}\label{bsdecontrollo_2}
    &R_{t}+\int_{t \wedge \tau}^{T \wedge \tau}\int_{\R^{d+ 1}} \Sigma_s(x_1,x_2)\,(\mu^Z(\rmd s,\rmd x_1,\rmd x_2)- \nu^Z(\rmd s,\rmd x_1,\rmd x_2)) \notag\\
&=\bar g(X_{T \wedge \tau}) +\int_{t \wedge\tau}^{T \wedge \tau} \bar f(s,X_s,\Sigma_s(\cdot))\,\rmd C^{\F, X}_s.
\end{align}
with
\begin{equation}\label{defhamiltonian2}
    \bar f(\omega,t,y_1,\theta(\cdot))=\inf_{u\in U}\Big\{
\bar l_t(\omega, y_1,u)+ \int_{\R^{d}} \theta(x_1,0) \, (\bar r_t
(\omega,x_1,u)-1)\,\phi_t^{\F, X}(\omega,\rmd x_1)\Big\}
\end{equation}
for every $\omega\in\Omega$, $t\in[0,T]$, $y_1 \in \R^d$,  and $\theta\in
\call^1(\R^{d},\mathcal B(\R^{d}), \phi_t(\omega,\rmd x_1, \rmd x_2))$.
\begin{assumption}\label{hyp:hamiltoniana_2_new}
For every $\Theta\in L^{1,0}(\mu^Z)$ there exists 
$\underline{\hat u}^\Theta \in \hat \calc$ 
 such that, for almost all $(\omega, t)$ with respect to the measure $d_1(\om, t)dC_t(\omega)\P(d\omega)$,
\begin{align}\label{minselector_bis_new}
    \bar f(\omega,t,X_{t-}(\omega),\Theta_t(\omega, \cdot))&=
\bar l_t(\om, X_{t-}(\omega),\underline{\hat u}^\Theta(\omega,t ))\\
&+ \int_{\R^{d+ 1}} \Theta_t(\omega, x_1, 0) \,
(\bar r_t
(\omega,x_1, \underline{\hat u}^\Theta(\omega,t))-1)\,\phi_t^X(\omega,\rmd x_1).\notag
\end{align}
\end{assumption}

In order to prove existence and uniqueness for BSDE \eqref{bsdecontrollo_2} we will use the following auxiliary equation: $\P$-a.s., for all $t\in [0,T]$,
\begin{align}\label{bsdecontrollo_barg}
    &\bar Y_{t}+\int_{t }^{T}\int_{\R^{d+ 1}} \bar \Theta_s(x_1,x_2)\,(\mu^Z- \nu^Z)(\rmd s,\rmd x_1,\rmd x_2) \notag\\
&=\bar g(X_{T \wedge \tau}) +\int_{t}^{T} \bar f (s,X_s,\bar \Theta_s(\cdot))1_{[0,T \wedge \tau]}(s)\,\rmd C^{\F, X}_s.
\end{align}
The proofs of the  two following results are postponed to Appendix \ref{A:B}.
\begin{proposition}\label{P:5.20}
	Let Assumptions \ref{A:appl}, \ref{A:appl2}, \ref{hyp:controllosingle}, \ref{hyp:bar_rl} and \ref{hyp:hamiltoniana_2_new} hold. Set
		\begin{align}\label{D:Lbis}
 \bar L:=\textup{ess}\sup_{\omega} \big( \sup\{ |\bar r_t (x_1, u)-1|\,:\,
  t\in [0,T], \,x_1 \in \R^d, \,u\in U\} \big),
\end{align}
and let  Assumption
\ref{hyp:controllo_barg_sigma}  hold true with $\beta >\bar L^2$.
Then BSDE \eqref{bsdecontrollo_barg} admits a unique solution $(\bar Y, \bar \Theta(\cdot))\in  L^{2, \beta}_{\text{Prog}}(\Omega \times [0,\,T],\Gbb) \times L^{2, \beta}(\mu^Z,\Gbb)$.

\end{proposition}

\begin{theorem}\label{T:new}
Let Assumptions \ref{A:appl}, \ref{A:appl2}, \ref{hyp:controllosingle}, \ref{hyp:bar_rl}
and \ref{hyp:hamiltoniana_2_new}
 hold true. Assume also that Assumptions  and \ref{hyp:controllo_barg_sigma} hold true with $\beta >\bar L^2$, with $\bar L$ in \eqref{D:Lbis}, and let  $(\bar Y, \bar \Theta(\cdot)) \in  L^{2, \beta}_{\text{Prog}}(\Omega \times [0,\,T],\Gbb) \times L^{2, \beta}(\mu^Z,\Gbb)$ denote
 the unique solution to  BSDE
  \eqref{bsdecontrollo_barg}.
  Then
  the BSDE \eqref{bsdecontrollo_2}
admits a unique solution given by $(R, \Sigma)=(\bar Y_{\cdot \wedge  \tau},  \bar \Theta 1_{[0,T \wedge \tau]})
$.
In particular, $\bar Y=\bar Y_{\cdot \wedge \tau}$, $
\P(\rmd \omega)$-a.e. and $\bar \Theta = \bar \Theta 1_{[0,T \wedge \tau]}$, $\phi_t(\omega, \rmd x_1, \rmd x_2)\,\rmd C_t(\omega) \,\P(\rmd \omega)$-a.e.
\end{theorem}

\textbf{Solution to the optimal control problem.}
We can then give following result, that is the analogous  of Theorem \ref{teoremacontrollononmarkov} in the present framework.
\begin{theorem}\label{T:cont2_bis}
Let  Assumptions \ref{A:appl}  \ref{A:appl2}, \ref{hyp:controllosingle}, \ref{hyp:bar_rl} and
\ref{hyp:hamiltoniana_2_new} hold true. Let also Assumption \ref{hyp:controllo_barg_sigma}  hold true
with $\beta>\bar L^2$, with $\bar L$ in \eqref{D:Lbis}, and   condition \eqref{C_r} hold true with $M_{\bar r}$ in place of $M_{r}$. Let  $(\bar Y, \bar \Theta)=(\bar Y_{\cdot \wedge  \tau},  \bar \Theta 1_{[0,T \wedge \tau]})\in L^{2, \beta}_{\text{Prog}}(\Omega \times [0,\,T],\Gbb) \times L^{2, \beta}(\mu^Z,\Gbb)$  denote
 the unique  solution to  BSDE \eqref{bsdecontrollo_2}, with
 corresponding admissible control $\underline{\hat u}^\Theta \in \hat\calc$  satisfying \eqref{minselector_bis_new}.
Then
$\underline{\hat u}^\Theta$ is   optimal and $\bar Y_0$ is the optimal cost, i.e.
$$
\bar Y_0= \bar J(\underline{\hat u}^\Theta)=
\inf_{\hat u\in\hat \calc }\bar J(\hat u).
$$
\end{theorem}
\proof
The proof 
consists once again  in proving the  fundamental relation.
By the analogous result of Lemma \ref{L:Lu} for the present framework,  for every $\hat u \in \hat \calc$, we have	$\sup_{t \in [0,\,T]} \E [|L^{\hat u}_t|^2] < \infty$ and $\E [L^{\hat u}_T ] = 1$.  In particular, $L^{\hat u}$ is a square integrable $\Gbb$-martingale for every $\hat u\in \hat \calc$.
Let $\hat u \in \hat \calc$ be fixed.
Proceeding as in \eqref{GintPu}, we see that $\Ebb_{\hat u}[|\bar g(X_T)|]<+\infty$, while
H\"older inequality and Assumption \ref{hyp:rl} yield $\Theta(\cdot) \in L^{1,0}(\mu^Z,\Gbb)$ under $\P_{ \hat u}$.
Setting $t=0$ and taking the expectation $\E_{ \hat u}[\cdot]$ in BSDE \eqref{bsdecontrollo_2}, we get
\begin{align*}
    &\bar Y_0+\E_{ \hat u}\big[\int_0^{T\wedge \tau}\int_{\R^{d}} \bar \Theta_s(x_1,0) \, (\bar r_s
(x_1,\hat u)-1)\,\phi^{\Gbb, X}_s(\rmd x_1)\rmd C^{\Fbb, X}_s\big]\\
&=
\E_{ \hat u}[\bar g(X_{T \wedge \tau})] +\E_{ \hat u}\big[\int_0^{T\wedge \tau} \bar f(s,X_s,\bar \Theta_s(\cdot))\,\rmd C^{\Fbb, X}_s\big].
\end{align*}
Then, adding and subtracting $\E_u\big[\int_0^{T\wedge \tau} \bar l_s(X_s,\hat u)\,\rmd C^{X,\Fbb}_s\big]$ we obtain
\begin{align}\label{fund_rel_2}
    &\bar Y_0= \bar J(u) \notag\\
&+\E_u\Big[\int_0^{T\wedge \tau} \big[\bar f(s,X_{s},\bar \Theta_s(\cdot))-\bar l(s,X_{s},\hat u) -\int_{\R^{d}} \bar \Theta_s(x_1,0) \, (\bar r_s
(x_1,u)-1)\,\phi^{\Fbb, X}_s(\rmd x_1) \big]\rmd C^{\Fbb, X}_s\Big]
\end{align}
where we have also used the continuity of $C$. The conclusion follows from the definition of $\bar f$ in \eqref{minselector_bis_new}, noticing that the term in the square brackets in \eqref{fund_rel_2} is non positive, and it equals $0$ if $\hat u(\cdot) = \underline{\hat u}^\Theta(\cdot)$.
\endproof

\subsection{Relationship between the two control problems} \label{S:relcontrpb}
We end this section with some additional considerations on the  two  optimal control problems \eqref{contr_pb} and  \eqref{contr_pb_2}.  We first notice that
	the functional cost
	in optimal control problem  in  \eqref{contr_pb_2} can be equivalently rewritten in the form
	\begin{equation*}
\bar J(\hat u)= \E_{\hat u}\left[
\int_0^{T} \bar l_t(X_t,\hat u_t) \,1_{[0,T \wedge \tau(\omega)]}(t)\rmd C^{\F, X}_t + \bar g(X_{T\wedge \tau})\right],\quad \hat u \in \hat{\calc}.
\end{equation*}
 Clearly
 $\bar l1_{[0,T\wedge\tau]}$ is $\calp(\Gbb)\otimes \mathcal B(\R^{d})\otimes \calu$-measurable and $\bar g(X_{T\wedge \tau})$ is $\calg_{T }$-measurable. Recalling  Remark \ref{equivdec}, we see that the control problem in  \eqref{contr_pb_2}  can be seen as  one   of the type studied in Section \ref{S:T}, where however a subclass  of admissible controls is considered.

Let us now consider the \emph{enlarged} optimal control problem obtained from \eqref{contr_pb_2} by  taking the infimum over all the $\Gbb$-predictable processes $u(\cdot)$:
\begin{equation}\label{contr_pb_2bis}
\inf_{u \in \calc}\bar J(u)= \E_{u}\left[
\int_0^{T} \bar l_t(X_t,u_t) \,1_{[0,T \wedge \tau(\omega)]}(t)\rmd C^{\F, X}_t + \bar g(X_{T\wedge \tau})\right].
\end{equation}
According to Section \ref{S:T}, one can solve optimal control problem \eqref{contr_pb_2bis} by considering the following  BSDE: $\P$-a.s., for all $t\in [0,T]$,
\begin{align}\label{BSDE1bis}
    Y_{t}+\int_{t }^{T}\int_{\R^{d+ 1}} \Theta_s(x_1,x_2)\,(\mu^Z- \nu^Z)(\rmd s,\rmd x_1,\rmd x_2) =\bar g(X_{T \wedge \tau}) +\int_{t}^{T} \, f(s,X_s,\Theta_s(\cdot))\,\rmd C^{\F, X}_t,
\end{align}
where
\begin{align*}
f(\omega,t,y_1,\theta(\cdot))
&:=\inf_{u\in U}\Big\{
\bar l_t(\omega, y_1,u)1_{[0, T \wedge \tau(\omega)]}(t)+ \int_{\R^{d}} \theta(x_1,0)\, (\bar r_t
(\omega,x_1,u)-1)\,\phi_t^{\F, X}(\omega,\rmd x_1)\Big\}
\end{align*}
for every $\omega\in\Omega$, $t\in[0,T]$, $y_1 \in \R^d$,  and $\theta\in
\call^1(\R^{d + 1},\mathcal B(\R^{d + 1}), \phi_t(\omega,\rmd x_1, \rmd x_2))$.
\begin{assumption}\label{hyp:hamiltoniana_2}
For every $\Theta\in L^{1,0}(\mu^Z)$ there exists a $\Gbb$-predictable process
$\underline{u}^\Theta:\Omega\times [0,T]\to U$,  such that,
 for almost all $(\omega, t)$ with respect to the measure $d_1(\om, t)dC_t(\omega)\P(d\omega)$,
\begin{align}\label{ftheta}
    f(\omega,t,X_{t-}(\omega),\Theta_t(\omega, \cdot))&=
\bar l_t(\om, X_{t-}(\omega),\underline{u}^\Theta(\omega,t ))1_{[0, T \wedge \tau(\omega)]}(t) \\
&+ \int_{\R^{d+ 1}} \Theta_t(\omega, x_1, 0) \,
(\bar r_t
(\omega,x_1, \underline{u}^\Theta(\omega,t))-1)\,\phi_t^{\F, X}(\omega,\rmd x_1).\notag
\end{align}
\end{assumption}

Under Assumption \ref{hyp:hamiltoniana_2},  the well-posedness of BSDE \eqref{BSDE1bis} directly follows from  Proposition \ref{Pwellpos1}, noticing that  $l_t(\om, x_1, u):= \bar l_t(\om, x_1, u) 1_{[0, T \wedge \tau(\om)]}(t)$ and $g(\om, x_1):=g(\om) = \bar g(X_{T \wedge \tau}(\om))$ satisfy respectively Assumption \ref{hyp:rl} with $M_l = M_{\bar l}$ and Assumption \ref{hyp:controllo_g}   with $\beta >\bar L^2$.
Then Theorem  \ref{teoremacontrollononmarkov} reads in the present framework as follows.

 \begin{theorem}\label{T:cont2}
Let  Assumptions \ref{A:appl}  \ref{A:appl2}, \ref{hyp:controllosingle}, \ref{hyp:bar_rl} and
\ref{hyp:hamiltoniana_2} hold true. Assume also that Assumption \ref{hyp:controllo_barg_sigma}  holds true
with $\beta>\bar L^2$, with $\bar L$ in \eqref{D:Lbis}, and  that condition \eqref{C_r} holds true with $M_{\bar r}$ in place of $M_{r}$. Let  $(Y,\Theta(\cdot))\in L^{2, \beta}_{\text{Prog}}(\Omega \times [0,\,T],\Gbb) \times L^{2, \beta}(\mu^Z,\Gbb)$  denote
 the unique  solution to  BSDE \eqref{BSDE1bis}, with
 corresponding admissible control $\underline{u}^\Theta \in \mathcal C$  satisfying \eqref{ftheta}.
 Set
 \begin{equation*}
\bar J(u):= \E_{u}\left[
\int_0^{T} \bar l_t(X_t,\hat u_t) \,1_{[0,T \wedge \tau]}(t)\rmd C^{\F, X}_t + \bar g(X_{T\wedge \tau})\right],\quad u \in {\calc}.
\end{equation*}
Then
$$
Y_0= \bar J(\underline{u}^\Theta)=
\inf_{u\in\calc }\bar J(u).
$$
\end{theorem}
\endproof

Let us now go back to optimal control problem \eqref{contr_pb_2}.
In order to solve it,  we aim at finding an admissible  process $\hat {\underline{u}}\in \hat \calc$ such that
$$
\bar J(\underline{\hat u})=
\inf_{\hat u\in \hat \calc }\bar J(\hat u).
$$
We  prove below that
such an optimal control process exists and  is provided by
 $
 \hat {\underline{u}}^\Theta:= \underline{u}^\Theta 1_{[0, T \wedge \tau]}
 $
  with $\underline{u}^\Theta \in  \calc$ the process  given in  Theorem \ref{T:cont2}. Moreover, the value functions of the optimal control problems \eqref{contr_pb_2} and \eqref{contr_pb_2bis} coincide.

\begin{theorem}\label{P:5.22}
Let  Assumptions \ref{A:appl}  \ref{A:appl2}, \ref{hyp:controllosingle}, \ref{hyp:bar_rl} and
\ref{hyp:hamiltoniana_2} hold true. Assume also that Assumption \ref{hyp:controllo_barg_sigma}  holds true
with $\beta>\bar L^2$, with $\bar L$ in \eqref{D:Lbis}, and  that condition \eqref{C_r} holds true with $M_{\bar r}$ in place of $M_{r}$. Let  $(Y,\Theta(\cdot))\in L^{2, \beta}_{\text{Prog}}(\Omega \times [0,\,T],\Gbb) \times L^{2, \beta}(\mu^Z,\Gbb)$  denote
 the unique  solution to  BSDE \eqref{BSDE1bis}, with
 corresponding admissible control $\underline{u}^\Theta \in \mathcal C$  satisfying \eqref{ftheta}.
 Set
 $$
 \hat {\underline{u}}^\Theta:= \underline{u}^\Theta 1_{[0, T \wedge \tau]} \in \hat \calc.
 $$
Then $\hat {\underline{u}}^\Theta$ is an optimal control process for the control problem \eqref{contr_pb_2} and the value function is represented by the initial solution to BSDE \eqref{BSDE1bis}, namely
\begin{equation}\label{final_rep}
Y_0= \bar J(\underline{\hat u}^\Theta)=
\inf_{\hat u\in \hat \calc }\bar J(\hat u)=
\inf_{u\in \hat \calc }\bar J( u).
\end{equation}
\end{theorem}
\proof
We divise the proof into two steps.

\noindent \emph{Step 1.} The unique solution $(Y, \Theta(\cdot))$ to BSDE
\eqref{BSDE1bis}
coincides with the one to BSDE \eqref{bsdecontrollo_barg}. In particular,  $Y=Y_{\cdot \wedge \tau}$, $
\P(\rmd \omega)$-a.e. and $\Theta = \Theta 1_{[0,T \wedge \tau]}$ , $\phi_t(\omega, \rmd x_1, \rmd x_2)\,\rmd C_t(\omega) \,\P(\rmd \omega)$-a.e.

\medskip

\noindent \emph{Step 2.}  Identity \eqref{final_rep} holds true. 
\medskip

\noindent \emph{Proof of Step 1.} Let $(\bar Y, \bar \Theta)$ be the unique solution to BSDE \eqref{bsdecontrollo_barg}.
It is enough to show the identity
\begin{align}\label{barf}
\bar f(\omega, s,X_s(\omega),\bar \Theta_s(\omega,\cdot))\,1_{[0,T \wedge \tau(\omega)]}(s)
&=f(\omega,s,X_s(\omega),\bar \Theta_s(\omega,\cdot)),\,\,  d_1(\om,s)\rmd C_s(\omega) \P(\rmd \omega)\text{-a.e.}
\end{align}
As a matter of fact, plugging \eqref{barf} in BSDE \eqref{bsdecontrollo_barg}, we would  get BSDE \eqref{BSDE1bis}.
Then, by the uniqueness of the solution, $(Y, \Theta(\cdot))$  
would coincide with the solution to BSDE \eqref{bsdecontrollo_barg}  and, by Theorem   \ref{T:new}, $Y=Y_{\cdot \wedge \tau}$, $
\P(\rmd \omega)$-a.e. and $\Theta = \Theta 1_{[0,T \wedge \tau]}$, $\phi_t(\omega, \rmd x_1, \rmd x_2)\,\rmd C_t(\omega) \,\P(\rmd \omega)$-a.e.

Let us thus prove \eqref{barf}. By Theorem \ref{T:new}, $\bar \Theta = \bar \Theta 1_{[0,T \wedge \tau]}$ , $\phi_t(\omega, \rmd x_1, \rmd x_2)\,\rmd C_t(\omega) \,\P(\rmd \omega)$-a.e., so that
\begin{align}\label{coll1}
f\Big(\omega,s,X_s(\omega),\bar \Theta_s(\omega,\cdot)\,1_{[0,T \wedge \tau(\omega)]}(s)\Big)
&=f(\omega,s,X_s(\omega),\bar \Theta_s(\omega,\cdot)).
\end{align}
On the other hand, recalling  \eqref{defhamiltonian2}, we have that, $\rmd C_s(\omega) \P(\rmd \omega)$-almost surely on $\Omega \times [0,\,T]$,
\begin{align}\label{coll2}
&f\Big(\omega,s,X_s(\omega),\bar \Theta_s(\omega,\cdot)\,1_{[0,T \wedge \tau(\omega)]}(s)\Big)\notag\\
  &=\inf_{u\in U}\Big\{
\bar l_s(\omega, X_s(\omega),u)\,1_{[0,T \wedge \tau(\omega)]}(s)+ \int_{\R^{d}} \bar \Theta(x_1,0)\,1_{[0,T \wedge \tau(\omega)]}(s) \, (\bar r_s
(\omega,x_1,u)-1)\,\phi_s^{\F, X}(\omega,\rmd x_1)\Big\}\notag\\
    &=1_{[0,T \wedge \tau(\omega)]}(t)\,\inf_{u\in U}\Big\{
\bar l_s(\omega, X_s(\omega),u)+ \int_{\R^{d}} \bar \Theta(x_1,0) \, (\bar r_s
(\omega,x_1,u)-1)\,\phi_s^{\F, X}(\omega,\rmd x_1)\Big\}\notag\\
&=\bar f(\omega,s,X_s(\omega),\bar \Theta_s(\omega,\cdot))\,1_{[0,T \wedge \tau(\omega)]}(s).
\end{align}
Collecting \eqref{coll1} and \eqref{coll2}, we get  \eqref{barf}.

\medskip

\noindent \emph{Proof of Step 2.} Let  $(Y,\Theta(\cdot))\in L^{2, \beta}_{\text{Prog}}(\Omega \times [0,\,T],\Gbb) \times L^{2, \beta}(\mu^Z,\Gbb)$  denote
 the unique  solution to  BSDE \eqref{BSDE1bis}, with
 corresponding admissible control $\underline{u}^\Theta \in \mathcal C$  satisfying \eqref{ftheta}.
Recalling the proof of Theorem \ref{T:cont2_bis}, it is enough to show that, for almost all $(\omega, t)$ such that $t \leq T \wedge \tau(\omega)$ with respect to the measure $d_1(\om, t)dC_t(\omega)\P(d\omega)$,
 \begin{align}\label{ftheta2}
    \bar f(\omega,t,X_{t-}(\omega), \Theta_t(\omega, \cdot))&=
\bar l_t(\om, X_{t-}(\omega), {\underline{u}}^\Theta(\omega,t )) \\
&+ \int_{\R^{d}} \Theta_t(\omega, x_1, 0) \,
(\bar r_t
(\omega,x_1, {\underline{u}}^\Theta(\omega,t))-1)\,\phi_t^{\F, X}(\omega,\rmd x_1).\notag
\end{align}
Identities   \eqref{barf} and \eqref{ftheta}, together with Step 1,  yield
\begin{align*}
&\bar f(\omega, t,X_t(\omega),\Theta_t(\omega,\cdot))\,1_{[0,T \wedge \tau(\omega)]}(t)\\
&=f(\omega,t,X_t(\omega),\Theta_t(\omega,\cdot))\\
&=
1_{[0,T \wedge \tau(\omega)]}(t) \bar l_t(\om, X_{t-}(\omega),\underline{u}^\Theta(\omega,t )) \\
&+ \int_{\R^{d}} \Theta_t(\omega, x_1, 0) \,
(\bar r_t
(\omega,x_1, \underline{u}^\Theta(\omega,t))-1)\,\phi_t^{\F, X}(\omega,\rmd x_1)\\
&=
1_{[0,T \wedge \tau(\omega)]}(t) \bar l_t(\om, X_{t-}(\omega),\underline{u}^\Theta(\omega,t )) \\
&+ \int_{\R^{d}} \Theta_t(\omega, x_1, 0) \,1_{[0,T \wedge \tau(\omega)]}(t)\,
(\bar r_t
(\omega,x_1, \underline{u}^\Theta(\omega,t))-1)\,\phi_t^{\F, X}(\omega,\rmd x_1)\\
&=1_{[0,T \wedge \tau(\omega)]}(t)\Big\{\,\bar l_t(\om, X_{t-}(\omega), {\underline{u}}^\Theta(\omega,t )) + \int_{\R^{d}} \Theta_t(\omega, x_1, 0) \,
(\bar r_t
(\omega,x_1, \underline{u}^\Theta(\omega,t))-1)\,\phi_t^{\F, X}(\omega,\rmd x_1)\Big\}
\end{align*}
that provides \eqref{ftheta2}.
\endproof

\appendix
\renewcommand\thesection{Appendix}
\section{}
\renewcommand\thesection{\Alph{subsection}}
\renewcommand\thesubsection{\Alph{subsection}}

\subsection{Proofs of technical results of Section \ref{sec:pred.proj}}\label{sec:app}

\begin{proof}[\textbf{Proof of Theorem \ref{T:Zcomp}}]
Let us start by proving (i). We have
\begin{align*}
	\mu^Z(\om, \rmd t, \rmd x_1, \rmd x_2)&=
	 \sum_{s >0} 1_{\{\Delta Z_s(\omega) \neq 0 \}}
	 \delta_{(s, \Delta Z_s(\omega))}(\rmd t, \rmd x_1,  \rmd x_2)\\
	&=\sum_{s >0} 1_{\{\Delta X_s(\omega) \neq 0, \Delta H_s(\omega) \neq 0 \}}
	 \delta_{(s, (\Delta X_s(\omega), \Delta H_s(\omega)))}(\rmd t, \rmd x_1,  \rmd x_2)\\
	 &+\sum_{s >0} 1_{\{\Delta X_s(\omega) \neq 0, \Delta H_s(\omega) = 0 \}}
	 \delta_{(s, (\Delta X_s(\omega), 0))}(\rmd t, \rmd x_1,  \rmd x_2)\\
	 &+\sum_{s >0} 1_{\{\Delta X_s(\omega) = 0, \Delta H_s(\omega) \neq 0 \}}
	 \delta_{(s, (0, \Delta H_s(\omega)))}(\rmd t, \rmd x_1,  \rmd x_2).
\end{align*}
Now we notice that, since by assumption $\Delta X \Delta H =0$, we have $\{\Delta X \neq 0\} \cap \{\Delta H \neq 0\} = \emptyset$. Therefore, noticing that moreover $\{\Delta X \neq 0\} \subseteq \{\Delta H= 0\}$ and $\{\Delta H \neq 0\} \subseteq \{\Delta X= 0\}$, previous expression reads
\begin{align*}
	\mu^Z(\om, \rmd t, \rmd x_1, \rmd x_2)	&=\sum_{s >0} 1_{\{\Delta X_s(\omega) \neq 0\}}
	 \delta_{(s, \Delta X_s(\omega))}(\rmd t, \rmd x_1)\delta_0(\rmd x_2)\\
	 &+\sum_{s >0} 1_{\{\Delta H_s(\omega) \neq 0 \}}
	 \delta_{(s, \Delta H_s(\omega))}(\rmd t,   \rmd x_2)\delta_0(\rmd x_2)\\
	 &=\mu^X(\om, \rmd t, \rmd x_1)\delta_0(\rmd x_2) + \mu^H(\om, \rmd t, \rmd x_2) \delta_0(\rmd x_1).
\end{align*}
Let us now prove (ii). Set
\begin{align*}
	\nu^{\Gbb,Z}(\om, \rmd t, \rmd x_1, \rmd x_2):= \nu^{\Gbb,X}(\om, \rmd t, \rmd x_1)\delta_0(\rmd x_2)+\nu^{\Gbb,H}(\om, \rmd t, \rmd x_2)\delta_0(\rmd x_1).
\end{align*}
We have to prove that $\nu^{\Gbb}$ is the $\Gbb$-dual predictable projection of $\mu^Z$.  To this end it is sufficient to show that, for every $\Gbb$-predictable function $W$ satisfying   $W \geq  0$ and $W \ast \mu^Z \in \Ascr_\mathrm{loc}^+(\Gbb)$, the process $W\ast\mu\p Z-W\ast\nu\p\Gbb\in\Mloc(\Gbb)$. So, let us consider such a $\Gbb$-predictable function $W$. We then have

\begin{align*}
	W \ast \mu^{Z}_t= \int_0^t \int_{\R^{d}}W(\om, s, x_1, 0)\mu^{X}(\om, \rmd s, \rmd x_1)+\int_0^t \int_{\R^{\ell}}W(\om, s, 0, x_2)\mu^{H}(\om, \rmd s, \rmd x_2)
\end{align*}
 and
\begin{align*}
	W \ast \nu^{\Gbb}_t= \int_0^t \int_{\R^{d}}W(\om, s, x_1, 0)\nu^{\Gbb, X}(\om, \rmd s, \rmd x_1)+\int_0^t \int_{\R^{\ell}}W(\om, s, 0, x_2)\nu^{\Gbb, H}(\om, \rmd s, \rmd x_2).
\end{align*}
Since $W \ast \mu^Z \in \Ascr_\mathrm{loc}^+(\Gbb)$  and $W\geq0$, we get
$$
\int_0^\cdot \int_{\R^{d}}W( s, x_1, 0)\mu^{X}(\rmd s, \rmd x_1), \int_0^\cdot \int_{\R^{\ell}}W(s, 0, x_2)\mu^{H}(\rmd s, \rmd x_2)\in \Ascr_\mathrm{loc}^+(\Gbb)
$$
and therefore
$$
\int_0^\cdot \int_{\R^{d}}W( s, x_1, 0)\nu^{\Gbb, X}(\rmd s, \rmd x_1), \int_0^\cdot \int_{\R^{\ell}}W(s, 0, x_2)\nu^{\Gbb, H}(\rmd s, \rmd x_2)\in \Ascr_\mathrm{loc}^+(\Gbb)
$$
that yields $W \ast \nu^{\Gbb} \in \Ascr_\mathrm{loc}^+(\Gbb)$. It remains to show  that $ W \ast \mu^Z-W \ast \nu^{\Gbb} \in \Mloc(\Gbb)$. By definition of $\nu^\Gbb$ we have
\begin{align*}
	W \ast \mu^{Z}-W \ast \nu^{\Gbb}&= \int_0^\cdot \int_{\R^{d}}W(\om, s, x_1, 0)\mu^{X}(\om, \rmd s, \rmd x_1)-\int_0^\cdot \int_{\R^{d}}W(\om, s, x_1, 0)\nu^{\Gbb, X}(\om, \rmd s, \rmd x_1)\\
	&+\int_0^\cdot \int_{\R^{\ell}}W(\om, s, 0, x_2)\mu^{H}(\om, \rmd s, \rmd x_2)-\int_0^\cdot  \int_{\R^{\ell}}W(\om, s, 0, x_2)\nu^{\Gbb, H}(\om, \rmd s, \rmd x_2).
\end{align*}
 By linearity, it follows that $W \ast \mu^Z-W \ast \nu^{\Gbb} \in\Mloc(\Gbb)$, $\nu\p{\Gbb, X}$ and $\nu\p{\Gbb, H}$ being the $\Gbb$-compensator of $\mu\p X$ and $\mu\p H$, respectively. Let now $W$ be an arbitrary $\Gbb$-predictable function such that $|W|\ast \mu^Z\in \Ascr_\mathrm{loc}^+(\Gbb)$. Applying the previous step to $W^+$ and $W^-$ we get that $|W|\ast \nu^\Gbb\in \Ascr_\mathrm{loc}^+(\Gbb)$ and $W \ast \mu^Z-W \ast \nu^{\Gbb} \in \Mscr_\mathrm{loc}^+(\Gbb)$. The proof is complete.
\end{proof}

\begin{proof}[\textbf{Proof of Lemma \ref{lem:dec.pred.fun}}]
Let $W$ be a $\Gbb$-predictable bounded function of the form $W(\om,t,x)=f(x)J_t(\om)$, where $f$ is a bounded $\Bscr(\Rbb\p d)$-measurable function and $J$ is a bounded $\Gbb$-predictable process. Then, by \cite[Lemma 4.4 (b)]{Jeu80} we have $W(\om,t,x)1_{[0,\tau]}(\om,t)=f(x)\ol J_t(\om)1_{[0,\tau]}(\om,t)$, where $\ol J$ is an $\Fbb$-predictable bounded process. By a monotone class argument we get the statement for arbitrary bounded  $\Gbb$-predictable functions. Then, by approximation, we get the statement for arbitrary $\Gbb$-predictable functions. The proof is complete.
\end{proof}

\begin{proof}[\textbf{Proof of Proposition \ref{prop:qlc.stop}}]
Because of the $\Fbb$-quasi-left continuity of $X$,  $X\p{p,\Fbb}$ is an $\Fbb$-adapted continuous process (see \cite[Corollary 5.28 (3)]{HWY92}). Furthermore, by \cite[Theorem 5.1]{AJ17}, the process
\[
X\p\tau-(X\p{p,\Fbb})\p\tau-\int_0\p{\tau\wedge\cdot}\frac{1}{A_{s-}}\,\rmd\aPP{X-X\p{p,\Fbb}}{m}_s\p\Fbb
\]
is a $\Gbb$-local martingale. This means that $(X\p{p,\Fbb})\p\tau+\int_0\p{\tau\wedge\cdot}\frac{1}{A_{s-}}\,\rmd\aPP{X-X\p{p,\Fbb}}{m}_s\p\Fbb$ is the $\Gbb$-dual predictable projection of $X$. We denote this process by $(X\p{p,\Gbb})\p\tau$. Since $X$ is $\Fbb$-quasi-left continuous, $\aPP{X-X\p{p,\Fbb}}{m}\p\Fbb$ is a continuous process. Indeed, be the property of the dual predictable projection, for every $\Fbb$-predictable finite valued stopping time $\sigma$ we have
\[
\Delta \aPP{X-X\p{p,\Fbb}}{m}\p\Fbb_\sig=\Ebb[\Delta [X-X\p{p,\Fbb},m]_\sig|\Fscr_{\sig-}]=\Ebb[\Delta X_\sig\Delta m_\sig|\Fscr_{\sig-}]=0.
\]
Hence, by the predictable section theorem, $\aPP{X-X\p{p,\Fbb}}{m}\p\Fbb$ being $\Fbb$-predictable, $\Delta \aPP{X-X\p{p,\Fbb}}{m}\p\Fbb=0$ up to an evanescent set. Therefore,  $(X\p{p,\Fbb})\p\tau$ being continuous, we deduce that $(X\p{p,\Gbb})\p\tau$ is continuous as well. Hence, by \cite[Corollary 5.28 (3)]{HWY92}, $X\p\tau$ is $\Gbb$-quasi-left continuous. The proof is complete.
\end{proof}

\begin{proof}[\textbf{Proof of Theorem \ref{thm:com.G}}]
Let $(\om,t,x)\mapsto W(\om,t,x)$ be a $\Gbb$-predictable bounded and nonnegative function. We then have that that $W\ast\mu\p X$ is locally bounded, and hence belongs to $\Ascr\p+_\mathrm{loc}(\Gbb)$, because of the estimate $W\ast\mu\p X\leq cN\p X$, where $c>0$ is such that $W(\om,t,x)\geq c$ and $N\p X$ is the point process associated to $X$. Because of Lemma \ref{lem:dec.pred.fun} there exists a bounded $\Fbb$-predictable function $\ol W$ such that
$$
W=\ol W1_{[0,\tau]}+W1_{(\tau,+\infty)}.
$$
We now analyse separately the two integrals $\ol W1_{[0,\tau]}\ast\mu\p X$ and $W1_{(\tau,+\infty)}\ast\mu\p X$.

We start with $\ol W1_{[0,\tau]}\ast\mu\p X$. We observe that $\ol W\ast\mu\p X$ is locally bounded and hence it belongs to $\Ascr\p+_\mathrm{loc}(\Fbb)$, $\ol W$ being an $\Fbb$-predictable bounded function. Hence, $(\ol W\ast\mu\p X)\p{p,\Fbb}$ exists and is equal to $\ol W\ast\nu\p{\Fbb, X}$. So, the process $\ol W\ast\mu\p X-\ol W\ast\nu\p{\Fbb, X}$ is an $\Fbb$-local martingale and, by \cite[Theorem 5.1]{AJ17},
\[
\ol W1_{[0,\tau]}\ast\mu\p X-\ol W1_{[0,\tau]}\ast\nu\p{\Fbb, X}-\int_0\p{\tau\wedge\cdot}\frac1{A_s-}\rmd\aPP{\ol W\ast\mu\p X-\ol W\ast\nu\p{\Fbb, X}}{m}\p\Fbb
\]
is a $\Gbb$-local martingale.  Since $m$ is an $\Fbb$-martingale, we find an $\Fbb$-predictable function $W\p m$ such that $m=m_0+W\p m\ast\mu\p X-W\p m\ast\nu\p{\Fbb,X}$. Furthermore, $m$ is in the class $BMO$. So, the $\Fbb$-predictable covariation $\aPP{\ol W\ast\mu\p X-\ol W\ast\nu\p{\Fbb, X}}{m}\p\Fbb$ is well defined and satisfies $\aPP{\ol W\ast\mu\p X-\ol W\ast\nu\p{\Fbb, X}}{m}\p\Fbb=\ol WW\p m\ast\nu\p{\Fbb,X}$, where we used \cite[Theorem II.1.13]{JS00} and that $\nu\p{\Fbb,X}$ is non-atomic in $t$, $X$ being $\Fbb$-quasi-left continuous (see \cite[Proposition II.2.9]{JS00}). So, by linearity we deduce that
\[
\ol W1_{[0,\tau]}\ast\mu\p X-\ol W\Big(1+\frac{W\p m}{A_-}\Big)1_{[0,\tau]}\ast\nu\p{\Fbb, X}
\]
is a $\Gbb$-local martingale. We also have that $\Delta m=W\p m(\cdot,\cdot,\Delta X)1_{\{\Delta X\neq0\}}=\Delta A$.  Therefore, we obtain
\[
A_-+W\p m(\cdot,\cdot,\Delta X)1_{\{\Delta X\neq0\}}=A_-+\Delta A=A\geq0.
\]

We now introduce the $\wt\Pscr(\Fbb)$-measurable set $D:=\{(\om,t,x)\in\wt\Om: A_{t-}(\om)+W\p m(\om,t,x)<0\}$. We denote by $M_{\mu\p X}$ the Dol\'eans measure induced by $\mu\p X$, that is, $M_{\mu\p X}(B)=\Ebb[1_B\ast\mu\p X_\infty]$, for every $B\in\Fscr\otimes\Bscr(\Rbb_+)\otimes\Bscr(\Rbb\p d)$. We then have $M_{\mu\p X}(D)=0$. Therefore, we can define the $\Fbb$-predictable function $W\p\prime(\om,t,x):= W\p m(\om,t,x)1_{D\p c}(\om,t,x)$ which again satisfies $m=W\p\prime\ast\mu\p X-W\p\prime\ast\nu\p{\Fbb,X}$ and  moreover $A_{t-}(\om)+W\p\prime(\om,t,x)\geq0$ \emph{identically}. We now define the $\Gbb$-predictable measure
\[
\nu\p{\Gbb,\leq\tau}(\om,\rmd t,\rmd x)=1_{[0,\tau]}\Big(1+\frac{W\p\prime(\om,t,x)}{A_{t-}(\om)}\Big)\nu\p{\Fbb,X}(\om,\rmd t,\rmd x).
\]
We then clearly have that $W1_{[0,\tau]}\ast\mu\p X-W\ast\nu\p{\Gbb,\leq\tau}$ is a $\Gbb$-local martingale.

We now come to the integral $W1_{(\tau,+\infty)}\ast\mu\p X$. To begin with, we introduce  the filtration $\Gbb\p\tau=(\Gscr_t\p\tau)_{t\geq0}$ by $\Gscr_t\p\tau:=\bigcap_{\ep>0}\Fscr_{t+\ep}\vee\sig(\tau)$, that is, $\Gbb\p\tau$ is the initial enlargement of $\Fbb$ by $\tau$. It is clear that $\Gbb\subseteq\Gbb\p\tau$ and that $\Gbb=\Gbb\p\tau$ over the stochastic interval $(\tau,+\infty]$. Following the proof of \cite[Proposition 3.14 and Theorem 4.1]{J85} we can show that there exists a $\Gbb\p\tau$-predictable function $U$ such that $1+U\geq0$ and
\[
\wt\nu(\om,\rmd t,\rmd x)=(1+U(\om,t,x))\nu\p{\Fbb,X}(\om,\rmd t,\rmd x)
\]
is the $\Gbb\p\tau$-dual predictable projection of $\mu\p X$. In particular, $W1_{(\tau,+\infty)}\ast\mu\p X$ being $\Gbb\p\tau$-adapted and locally bounded, we deduce that $W1_{(\tau,+\infty)}\ast\mu\p X-W1_{(\tau,+\infty)}\ast\wt\nu$ is a $\Gbb\p\tau$ local martingale. We now observe that the function $1_{(\tau,+\infty)}(1+U)$ is indeed $\Gbb$-predictable, since $\Gbb$ and $\Gbb\p\tau$ coincides over $(\tau,+\infty]$. This implies that the $\Gbb\p\tau$- local martingale $W1_{(\tau,+\infty)}\ast\mu\p X-W1_{(\tau,+\infty)}\ast\wt\nu$ is actually $\Gbb$-adapted. Furthermore, this is a martingale with bounded jumps, the process $W1_{(\tau,+\infty)}\ast\wt\nu$ being continuous and $W1_{(\tau,+\infty)}$ being bounded. We can therefore apply \cite[Proposition 9.18 (iii) and the subsequent comment]{J79} to obtain that $W1_{(\tau,+\infty)}\ast\mu\p X-W1_{(\tau,+\infty)}\ast\wt\nu$ is indeed a $\Gbb$-local martingale. We now define the $\Gbb$-predictable random measures $\nu\p{\Gbb,>\tau}(\om,\rmd t,\rmd x):=1_{(\tau,+\infty)}(\om,t)(1+U(\om,t,x))\nu\p{\Fbb, X}(\om,\rmd t,\rmd x)$ and
\[\begin{split}
\nu\p\Gbb(\om,\rmd t,\rmd x)&=\nu\p{\Gbb,\leq\tau}(\om,\rmd t,\rmd x)+\nu\p{\Gbb,>\tau}(\om,\rmd t,\rmd x)
\\&=
\bigg(1_{[0,\tau]}(\om,t)\Big(1+\frac{W\p\prime(\om,t,x)}{A_{t-}(\om)}\Big)+1_{(\tau,+\infty)}(\om,t)(1+U(\om,t,x))\bigg)\nu\p{\Fbb, X}(\om,\rmd t,\rmd x).
\end{split}
\]
Putting together the two previous steps, we get that the process $W\ast\mu\p X-W\ast\nu\p{\Gbb}$ is a $\Gbb$-local martingale, for every bounded nonnegative $\Gbb$-predictable function $W$.

Let now $W$ be a nonnegative $\Gbb$-predictable function and define $W\p n(\om,t,x):=W(\om,t,x)\wedge n$. Because of the previous step, the process $W\p n\ast\mu\p X-W\p n\ast\nu\p \Gbb$ is a $\Gbb$-local martingale. Let $(\sig_n)_n$ be a localizing sequence. For every $n\geq 0$ we get $\Ebb[W\p n1_{[0,\sig_n]}\ast\mu\p X_\infty]=\Ebb[W\p n1_{[0,\sig_n]}\ast\nu\p\Gbb_\infty]$. Since $W\p n1_{[0,\sig_n]}$ converges monotonically to $W$, by monotone convergence we obtain the identity $\Ebb[W\ast\mu\p X_\infty]=\Ebb[W\ast\nu\p\Gbb_\infty]$, for every nonnegative $\Gbb$-predictable function $W$. By \cite[Theorem II.1.18 (i)]{JS00} we deduce the identity $\nu\p{\Gbb}=\nu\p{\Gbb, X}$. In particular, since $\nu\p{\Fbb, X}(\{t\}\times\Rbb\p d)=0$ for every $t$, $X$ being $\Fbb$-quasi-left continuous, we deduce that  $\nu\p{\Gbb, X}(\{t\}\times\Rbb\p d)=0$ for every $t$, meaning that $X$ is also $\Gbb$-quasi left continuous. The proof of the theorem is complete.
\end{proof}

	\subsection{Proofs of technical results of Section \ref{sec:appl.con}}\label{A:B}
	\setcounter{equation}{0}
\setcounter{theorem}{0}

\proof [\textbf{Proof of Proposition \ref{Pwellpos1}}]
To show the result we apply \cite[Theorem 3.4]{CF13} to the present framework.
Let us then   check that all the hypotheses of \cite[Theorem 3.4]{CF13}   are satisfied. More precisely,
setting
$$
F(\omega, t, X_t(\omega), \Theta_t(\omega))=d_1(\omega, t)f(\omega, t, X_t(\omega), \Theta_t(\omega)),
$$
 with $d_1$ given in \eqref{d1}, we have to verify that:
\begin{enumerate}
	\item[(1)] The terminal cost $g(X_T)$ is $\mathcal G_T$-measurable and there exists $\beta >0$ such that $\E[\rme^{\beta C_T}|g(X_T)|^2]< \infty$ and $\E\Big[\int_0^T \rme^{\beta C_t}|F(t, X_t,  0)|^2\rmd C_t\Big] < \infty$.
	\item[(2)] For every $\omega\in\Omega$, $t\in[0,T]$,
$\Theta(\cdot) \in  L^{2, \beta}(\mu^Z,\Gbb)$, the mapping $(t, \omega)\mapsto 
F(\omega, t,X_t(\omega), \Theta_t(\omega, \cdot))$ is  $\mathbb G$-progressively measurable.
\item[(3)] For every $\omega \in \Omega$, $t\in[0,T]$,
	and $\zeta, \zeta'\in
\call^2(\R^{d + 1},\mathcal B(\R^{d + 1}), \phi_t(\omega,\rmd x_1, \rmd x_2))$
\begin{align}
 &
 |F(t, \omega,X_t(\omega),\zeta)-F(t, \omega,X_t(\omega),  \zeta')| \notag\\
 &\leq   L_F \Big(\int_{\R^{d + 1}} |\zeta(x_1,x_2) -\zeta'(x_1,x_2)| \phi_t(\omega,\rmd x_1,\rmd x_2)\Big)^{1/2},\label{Lipscond}
\end{align}
where $L_F>0$ is a constant.

\end{enumerate}

Point (1)  follows from Assumption \ref{hyp:controllo_g} with $\beta >L^2$,
being
\begin{align*}
  \E \Big[\int_0^T \rme^{\beta C_t}|F(t,X_t,0)|^2 
  \rmd C_t\Big]
  \leq \E \Big[\int_0^T \rme^{\beta C_t}|\inf_{u\in U}  l_t( X_t,u)|^2 \rmd C_t\Big]
  \le M_l^2\,\beta^{-1}\,
\E [\rme^{\beta C_T}] < + \infty.
\end{align*}
Concerning (3), by
the boundedness conditions  \eqref{ellelimitato} in  Assumption \ref{hyp:rl},
 it is easy to check  that estimate \eqref{Lipscond} holds with $L_F=L$.
 Finally,  the measurability requirements in (2) for the Hamiltonian $f$  hold thanks to  Assumption \ref{hyp:hamiltoniana}. 
By \eqref{minselector}, for every $\Theta(\cdot) \in L^{2, \beta}(\mu^Z,\Gbb)$ (recalling the inclusion $ L^{2,\beta}(\mu^Z)  \subseteq L^{1,0}(\mu^Z)$ for all $\beta>0$),
the map $(t, \omega) \mapsto    f(\omega,t,X_{t-}(\omega),\Theta_t(\omega, \cdot))$ is $\mathbb G$-progressively measurable; since by Remark \ref{R:contC} the process $C$ 
is   continuous and $X$ has piecewise constant paths,  the same holds  (after modification on a set of measure zero) for $(\omega, t) \mapsto    F(\omega,t,X_{t}(\omega),\Theta_t(\omega, \cdot))$.
\endproof

\proof  [\textbf{Proof of Proposition \ref{P:5.20}}]
 It suffices to verify that all the hypotheses of \cite[Theorem 3.4]{CF13}   are satisfied.
By Assumption \ref{hyp:controllo_barg_sigma} with $\beta >\bar L^2$, we have that  $\bar g(X_{T \wedge \tau})$ is $\mathcal G_{T}$-measurable and  $\E[\rme^{\beta C_{T}}|\bar g(X_{T \wedge \tau})|^2]< \infty$.
Moreover, by Assumption \ref{hyp:hamiltoniana_2_new}, for every $\omega\in\Omega$, $t\in[0,T]$,
$\Theta(\cdot) \in L^{2, \beta}(\mu^Z)$,
the mapping
$$
\bar F(\omega, t, X_t(\omega), \Theta_t(\omega))=d_1(\omega, t)\bar f(\omega, t, X_t(\omega), \Theta_t(\omega))1_{[0,T \wedge \tau(\om)]}(t)
$$
  is  $\mathbb G$-progressively measurable. Finally,  by Assumptions  \ref{hyp:bar_rl}
	and    \ref{hyp:controllo_barg_sigma}, $\E\Big[\int_0^T \rme^{\beta C_t}|\bar F(t, X_t, 0)|^2
	\, \rmd C_t\Big]$ is finite for every $\beta >\bar L^2$, and
	\begin{align*}
 &
 |\bar F(t, \omega,X_t(\omega),\zeta)-\bar F(t, \omega,X_t(\omega), \zeta')| \leq   \bar L \Big(\int_{\R^{d + 1}} |\zeta(x_1,x_2) -\zeta'(x_1,x_2)| \phi_t(\omega,\rmd x_1,\rmd x_2)\Big)^{1/2}.
\end{align*}
\endproof

\proof [\textbf{Proof of Theorem \ref{T:new}}]
Let $(\bar Y, \bar \Theta(\cdot))$ be the unique solution to BSDE \eqref{bsdecontrollo_barg}.
We know that the process $\bar Y$ is defined as
\begin{equation}\label{Y def}
\bar Y_t=\E\Big[\bar g(X_{T \wedge \tau}) + \int_0^T \bar f(s,X_s,\bar \Theta_s(\cdot))\,1_{[0,T \wedge \tau]}(s)\,\rmd C^{\F, X}_s \Big|{\mathcal G}_t\Big] - \int_0^t \bar f(s,X_s,\bar \Theta_s(\cdot))\,1_{[0,T \wedge \tau]}(s)\,\rmd C^{\F, X}_s
\end{equation}
and moreover the following representation holds:
\begin{align*}
&\bar g(X_{T \wedge \tau}) +\int_{0}^{T }  \bar f(s,X_s,\bar \Theta_s(\cdot))\,1_{[0,T \wedge \tau]}(s)\rmd C^{\F, X}_s
\\&= \E\Big[\bar g(X_{T \wedge \tau}) +\int_{0}^{T}  \bar f(s,X_s,\bar \Theta_s(\cdot))\,1_{[0,T \wedge \tau]}(s)\,\rmd C^{\F, X}_s \Big]
+\int_{0}^{T}\int_{\R^{d+ 1}} \bar \Theta_s(x_1,x_2)\,(\mu^Z- \nu^Z)(\rmd s,\rmd x_1,\rmd x_2)
\end{align*}
where the last integral is a martingale.
Consequently, for all $t \in [0,T]$,
\begin{align*}
&\bar Y_t =
 \E\Big[\bar g(X_{T \wedge \tau}) +\int_{0}^{T}  \bar f(s,X_s,\bar \Theta_s(\cdot))\,1_{[0,T \wedge \tau]}(s)\,\rmd C^{\F, X}_s \Big] \\
&+ \int_{0}^{t}\int_{\R^{d+ 1}} \bar \Theta_s(x_1,x_2)\,(\mu^Z- \nu^Z)(\rmd s,\rmd x_1,\rmd x_2)  - \int_0^t \bar f(s,X_s,\bar \Theta_s(\cdot))\,1_{[0,T \wedge \tau]}(s)\,\rmd C^{\F, X}_s.\notag
\end{align*}
 By Doob's stopping theorem and \eqref{bsdecontrollo_barg}, previous expression gives
\begin{align}\label{eq2}
 &\bar Y_{t \wedge \tau}
= \E\Big[\bar g(X_{T \wedge \tau}) +\int_0^{T\wedge \tau}  \bar f(s,X_s,\bar \Theta_s(\cdot))\,\rmd C^{\F, X}_s  \Big ]\\
&+\int_{0}^{t \wedge \tau}\int_{\R^{d+ 1}}
\bar \Theta_s(x_1,x_2)\,(\mu^Z- \nu^Z)(\rmd s,\rmd x_1,\rmd x_2) - \int_0^{t\wedge \tau} \bar f(s,X_s,\bar \Theta_s(\cdot))\,\rmd C^{\F, X}_s.\notag
\end{align}
Now we notice that for $t=0 $  and $t = T \wedge \tau$ in \eqref{Y def} we obtain respectively
$$
\bar Y_0= \E\Big[\bar g(X_{T \wedge \tau}) +\int_0^{T\wedge \tau}  \bar f(s,X_s,\bar \Theta_s(\cdot))\,\rmd C^{\F, X}_s  \Big ], \quad \bar Y_{T\wedge \tau} = \bar g(X_{T \wedge \tau}).
$$
Then,
\begin{align*}
&\bar g(X_{T \wedge \tau}) +\int_{t \wedge \tau}^{T\wedge \tau}
 \,\bar f(s,X_s,\bar \Theta_s(\cdot))
 \,\rmd C^{\F, X}_s\\
&= \bar Y_{t \wedge \tau}+\int_{t \wedge \tau}^{T \wedge \tau}\int_{\R^{d+ 1}} 
\bar \Theta_s(x_1,x_2)\,(\mu^Z- \nu^Z)(\rmd s,\rmd x_1,\rmd x_2)
\end{align*}
or, equivalently,
\begin{align*}
&\bar g(X_{T \wedge \tau}) +\int_{t \wedge \tau}^{T\wedge \tau}
 \,\bar f(s,X_s,\bar \Theta_s(\cdot)
 1_{[0,T \wedge \tau]}(s)
 )
 \,\rmd C^{\F, X}_s\\
&=\bar  Y_{t \wedge \tau}+\int_{t \wedge \tau}^{T \wedge \tau}\int_{\R^{d+ 1}} 1_{[0,T \wedge \tau]}(s)
\bar \Theta_s(x_1,x_2)\,(\mu^Z- \nu^Z)(\rmd s,\rmd x_1,\rmd x_2).
\end{align*}
that proves that  $(\bar Y_{\cdot \wedge  \tau}, \bar \Theta(\cdot) 1_{[0,T \wedge \tau]})$ satisfies \eqref{bsdecontrollo_2}.

At this point, we notice that
		\begin{align*}
    &\bar Y_{t \wedge \tau}+\int_{t \wedge \tau}^{T \wedge \tau}\int_{\R^{d+ 1}} \bar \Theta_s(x_1,x_2)
    \,1_{[0,T \wedge \tau]}(s)
    \,(\mu^Z- \nu^Z)(\rmd s,\rmd x_1,\rmd x_2) \notag\\
&=\bar g(X_{T \wedge \tau}) +\int_{t\wedge \tau}^{T \wedge \tau} \,\bar f(s,X_s,\bar \Theta_s(\cdot)\,1_{[0,T \wedge \tau]}(s)
)\,
\rmd C^{\F, X}_s, \quad t \in [0,\,T],
\end{align*}
Therefore,
		\begin{align*}
    &\bar Y_{t \wedge \tau}+\int_{t}^{T}\int_{\R^{d+ 1}} \bar \Theta_s(x_1,x_2)
    \,1_{[0,T \wedge \tau]}(s)
    \,(\mu^Z- \nu^Z)(\rmd s,\rmd x_1,\rmd x_2) \notag\\
&=\bar g(X_{T \wedge \tau}) +\int_{t}^{T} \,\bar f(s,X_s,\bar \Theta_s(\cdot)\,1_{[0,T \wedge \tau]}(s)
)\,1_{[0,T \wedge \tau]}(s)\rmd C^{\F, X}_s, \quad t \in [0,\,T],
\end{align*}
namely the pair $(\bar Y_{\cdot \wedge \tau}, \bar \Theta\,1_{[0,T \wedge \tau]})$ solves BSDE   \eqref{bsdecontrollo_barg} as well.
The conclusion follows from the uniqueness of the solution to \eqref{bsdecontrollo_barg} in $ L^{2, \beta}_{\text{Prog}}(\Omega \times [0,\,T],\Gbb) \times L^{2, \beta}(\mu^Z,\Gbb)$.
\endproof


\begin{thebibliography}{19}
\providecommand{\natexlab}[1]{#1}
\providecommand{\url}[1]{\texttt{#1}}
\expandafter\ifx\csname urlstyle\endcsname\relax
  \providecommand{\doi}[1]{doi: #1}\else
  \providecommand{\doi}{doi: \begingroup \urlstyle{rm}\Url}\fi

\bibitem[Aksamit and Jeanblanc(2017)]{AJ17}
A.~Aksamit and M.~Jeanblanc.
\emph{Enlargement of filtration with finance in view}.
SpringerBriefs in Quantitative Finance,
Springer, 2017.


\bibitem[Aksamit and Choulli and  Jeanblanc(2018)]{ACJ}
 A.~Aksamit, T.~Choulli and M.~Jeanblanc.
 \newblock Thin times and random times' decomposition.
\emph{Electron. J. Probab.}, 26 (2021), Paper No. 31, 22 pp, DOI 10.1214/20-EJP569.

\bibitem[Aksamit et~al.(2019)Aksamit, Jeanblanc, and Rutkowski]{AJR18}
A.~Aksamit, M.~Jeanblanc and M.~Rutkowski.
Predictable representation property for progressive enlargements of a
  poisson filtration.
\emph{Stoch. Proc. Appl.}, 129\penalty0 (4):\penalty0 1229--1258,
  2019.

\bibitem[Aksamit, Li, and Rutkowski]{ALR21}
A.~Aksamit, L.~Li and M.~Rutkowski.
Generalized BSDEs with random time horizon in a progressively enlarged filtration, \emph{Preprint 2021},
 Arxiv version:  \url{http://arXiv.org/abs/2105.06654}.

\bibitem[Ankirchner, C. Blanchet-Scalliet, A. Eyraud-Loisel]{ABE09}
    S. Ankirchner, C. Blanchet-Scalliet, and A. Eyraud-Loisel.
 Credit risk premia and quadratic BSDEs with a single jump, \emph{Int. J. Theor. Appl. Finance}, 13\penalty0 (07):\penalty0 1103-1129, 2010.


\bibitem[Barlow(1978)]{ba:sf}
M.~T. Barlow.
Study of filtration expanded to include an honest time.
\emph{Z. Wahrscheinlichkeitstheorie verw. Gebiete}, 44:\penalty0 307--323, 1978.



\bibitem[Bielecki and Rutkowski (2002)]{BR02}
T. Bielecki  and M. Rutkowski.
\emph{Credit risk: modelling, valuation and hedging}.
Springer Finance, 2002.

\bibitem[Duffie D. and K. Singleton (2003)]{DS03}
D. Duffie  and K. Singleton. \emph{Credit risk: Pricing, measurement and management.}
Princeton University Press, 2003.





\bibitem[Calzolari and Torti(2021)]{CT21}
A.~Calzolari and B.~Torti.
Martingale representations in progressive enlargement by marked point processes. \emph{Preprint 2021}.
Arxiv version: \url{https://arxiv.org/abs/2107.04087}.

\bibitem[Confortola and Fuhrman(2013)]{CF13}
F. Confortola and M. Fuhrman.
\newblock Backward stochastic differential equations and optimal control of marked point processes.
\newblock \emph{SIAM J. Control Optim.}, 51\penalty0 (5):\penalty0 3592--3623, 2013.

\bibitem[Confortola and Fuhrman and Jacod(2016)]{CFJ16}
F. Confortola, M. Fuhrman and J. Jacod.
\newblock Backward stochastic differential equation driven by a marked point process: an elementary approach with an application to optimal control.
\newblock \emph{Ann. Appl. Probab.}, 26\penalty0 (3):\penalty0 1743--1773, 2016.

\bibitem[Alsheyab and Choulli(2021)]{AC21}
  S. Alsheyab and T. Choulli.  Reflected backward stochastic differential equations under stopping with an arbitrary random time.
 \emph{Preprint 2021}. Arxiv version: \url{https://arXiv.org/abs/2107.11896}.

\bibitem[Davis and Varaiya(1974)]{dv:mi}
M.~H.~A. Davis and P.~Varaiya.
\newblock On the multiplicity of an increasing family of sigma-fields.
\newblock \emph{Ann. Probab.}, 2:\penalty0 958--963, 1974.

\bibitem[Dellacherie(1972)]{D72}
C.~Dellacherie.
\emph{Capacit\'es et processus stochastiques}.
Springer, Berlin, 1972.

\bibitem[Di~Tella(2019)]{DT19}
P.~Di~Tella.
\newblock On the weak representation property in progressively enlarged
  filtrations with an application in exponential utility maximization.
\newblock \emph{Stoch. Proc. Appl.}, 130, 760--784, 2020.

\bibitem[Di~Tella, Engelbert (2021)]{DTE21}
P.~Di~Tella and H.-J.~Engelbert.
\newblock{Martingale Representation in Progressively Enlarged L\'evy Filtrations}.
\newblock\emph{Stochastics}, published online June 2021.

\bibitem[Di~Tella, Jeanblanc (2021)]{DTJ21}
P.~Di~Tella and  M.~Jeanblanc.
 Martingale representation in the enlargement of the filtration generated by a point process,
 	\emph{Stoch. Proc. Appl.}, 131\penalty0:\penalty0 103--121, 2021

\bibitem[Duffie(1986)]{du:se}
D.~Duffie.
 Stochastic equilibria: Existence, spanning number, and the no
  expected financial gain from Trade'hypothesis.
 \emph{Econometrica},
  1161--1183, 1986.

\bibitem[El Karoui, Jeanblanc, Jiao(2010)]{EJJ10}N.~El~Karoui, M.~Jeanblanc and  Y.~Jiao.
What happens after a default: The conditional density approach,
\emph{Stoch.\ Proc.\ Appl.}, 120\penalty0(7):\penalty0 1011--1032, 2010.

  \bibitem[W. Fleming and M. Soner (2006)]{FS06}
 W. Fleming and M. Soner.
 \emph{Controlled Markov Processes and Viscosity Solutions}.   Springer Verlag, 2006.


\bibitem[He et~al.(1992)He, J., and Yan]{HWY92}
S.~He, J.~Wang and J.~Yan.
\newblock \emph{Semimartingale theory and stochastic calculus}.
\newblock CRC, 1992.

\bibitem[Jacod(1979)]{J79}
J.~Jacod.
\newblock \emph{Calcul stochastique et probl{\`e}mes de martingales.}
\newblock Springer, 1979.

\bibitem[Jacod(1974/75)]{J74}
J.~Jacod.
Multivariate point processes: predictable projection, Radon-Nikodym derivatives, representation of martingales.
\emph{Z. Wahrsh. Verw. Gebiete}, 31\penalty0:\penalty0 235--253,  1974/75.

\bibitem[Jacod(1985)]{J85}
J.~Jacod.
\newblock{Grossissement initial, hypoth\`ese $(H\p\prime)$ et th\'eor\`eme de Girsanov}, p.\ 15-35 in
\newblock \emph{Grossissement de filtrations: examples et applications}, \newblock Lecture Notes in Mathematics, vol.\ 118, S\`eminaire de Calcul Stochastique 1982/83, Universit\'e Paris VI, Springer, 1985.
\newblock Springer, 1979.


\bibitem[Jacod and Shiryaev(2003)]{JS00}
J.~Jacod and A.~Shiryaev.
\newblock \emph{Limit theorems for stochastic processes}.
\newblock Springer,  2003.

%

\bibitem[Jeanblanc et~al.(2015)Jeanblanc, Mastrolia, D., and
  R{\'e}veillac]{JMPR15}
M.~Jeanblanc, T.~Mastrolia, D.~Possama{\"\i} and A.~R{\'e}veillac.
\newblock Utility maximization with random horizon: {A BSDE} approach.
\newblock \emph{International Journal of Theoretical and Applied Finance},
  18\penalty0 (07):\penalty0 1550045, 2015.


\bibitem[Jeanblanc Le Cam (2009)]{JLC09}
M.~Jeanblanc and  Y.~Le Cam.
Progressive enlargement of filtrations with initial times. \emph{Stoch. Proc. Appl.},
  119\penalty0 :\penalty0 2523--2543, 2009.




\bibitem[Jeulin(1980)]{Jeu80}
T.~Jeulin.
\emph{Semi-martingales et grossissement d'une filtration}.
Springer, 1980.

\bibitem[Jiao, Pham (2009) ]{JP09}
Y. Jiao and  H. Pham.
Optimal investment under counterparty risk: a default-density approach,
\emph{Finance Stoch.}, 2009.

\bibitem[Kharroubi, Lim, Ngoupeyou(2013)]{KLN13}
I.~Kharroubi, T.~Lim and  A.~Ngoupeyou.
 Mean-variance hedging on uncertain time horizon in a market with a jump,
 \emph{Appl. Math. Optim.}, 68\penalty0 (3):\penalty0 413--444, 2013.

\bibitem[Kharroubi, Lim, Ngoupeyou(2012)]{KL12}
I.~Kharroubi and  T.~Lim.
Mean-variance hedging on uncertain time horizon in a market with a jump,
 \emph{J. Theoret. Probab.}, 68\penalty0 (3):\penalty0 683--724, 2012.

\bibitem[Kusuoka(1999)]{Ku99}
S.~Kusuoka.
\newblock A remark on default risk models.
\newblock \emph{Adv. Math. Econ.}, 69--82. Springer,
  1999.



       \bibitem[Lim, Quenez (2008)]{LQ08}
  T. Lim and  M.C. Quenez.
Exponential  utility maximization in incomplete markets with defaults, \emph{Electron. J. Probab.}, 16\penalty0(53):\penalty0 143-1464, 2011.

\bibitem[Mansuy, Yor]{MY03}
R. Mansuy and M. Yor.
Random Times and Enlargements of Filtrations in Brownian Setting.
\emph{Lect. Notes in Mathematics}, Springer 2006.

 \bibitem [Pham (2010)]{P10}
 H. Pham.
Stochastic control under progressive enlargement of filtrations and applications to multiple defaults risk management.
 \emph{Stochastic Processes and their Applications} 120, 1795-1820, 2010.

\bibitem[Schonbucher (2003)]{S03}
P. Sch\"{o}nbucher.
\emph{Credit derivatives pricing models}.
Wiley Finance, 2003.
%

\bibitem[J. Yong and X.Y. Zhou (1999)]{YZ99}
    J. Yong and  X.Y. Zhou.
   \newblock  \emph{Stochastic Controls: Hamiltonian Systems and HJB Equations}. Springer Verlag, 1999.

\end{thebibliography}
\end{document}